\documentclass[reqno]{amsart}
\usepackage{amssymb,amsmath,hyperref}
\usepackage{amsrefs, float}
\usepackage[foot]{amsaddr}
\usepackage{bbold,stackrel, multicol}
 
\newcommand{\Z}{{\textsf{\textup{Z}}}}

\newtheorem{thm}{Theorem}

\newtheorem{cor}[thm]{Corollary}
\newtheorem{defi}[thm]{Definition}

\newtheorem{rem}[thm]{Remark}
\newtheorem{nota}[thm]{Notation}

\newtheorem{princ}[thm]{Principle}

\newtheorem{ack}[thm]{Acknowledgement}

\newtheorem*{tempo*}{Template}

\newcommand\be{\begin{equation}}
\newcommand\ee{\end{equation}} 

\usepackage{amsmath,amsfonts} 

\usepackage[applemac]{inputenc}

\def\bdefi{\begin{defi}\rm}
\def\edefi{\end{defi}}
\def\bnota{\begin{nota}\rm}
\def\enota{\end{nota}}

\def\FIVE{\Pi_{1}^{1}\text{-\textup{\textsf{CA}}}_{0}}
\def\SIX{\Pi_{2}^{1}\text{-\textsf{\textup{CA}}}_{0}}

\def\SIXK{\Pi_{k}^{1}\text{-\textsf{\textup{CA}}}_{0}^{\omega}}

\def\ATR{\textup{\textsf{ATR}}}

\def\PHP{\textup{\textsf{PHP}}}
\def\ZF{\textup{\textsf{ZF}}}
\def\ZFC{\textup{\textsf{ZFC}}}

\def\L{\textsf{\textup{L}}}

 \def\r{\mathbb{r}}

\def\RCA{\textup{\textsf{RCA}}}
\def\({\textup{(}}
\def\){\textup{)}}

\def\RCAo{\textup{\textsf{RCA}}_{0}^{\omega}}
\def\ACAo{\textup{\textsf{ACA}}_{0}^{\omega}}

\def\WKL{\textup{\textsf{WKL}}}

\def\WWKL{\textup{\textsf{WWKL}}}
\def\N{{\mathbb  N}}
\def\Q{{\mathbb  Q}}
\def\R{{\mathbb  R}}
\def\A{{\textsf{\textup{A}}}}

\def\SS{\textup{\textsf{S}}}

\def\di{\rightarrow}

\def\asa{\leftrightarrow}

\def\ACA{\textup{\textsf{ACA}}}

\def\QFAC{\textup{\textsf{QF-AC}}}

\def\SAC{\Sigma_{1}^{1}\textup{\textsf{-AC}}_{0}}

\def\BNC{\textup{\textsf{BNC}}}
\def\ENC{\textup{\textsf{ENC}}}

\def\SIND{\Sigma\textup{\textsf{-IND}}}

\def\OC{\textup{\textsf{OC}}}
\def\SSEP{\Sigma\textup{\textsf{-SEP}}}

\def\M{\mathcal{M}}

\def\SAC{\textup{\textsf{$\Sigma_{1}^{1}$-AC$_{0}$}}}

\def\cocode{\textup{\textsf{cocode}}}

\def\NIN{\textup{\textsf{NIN}}}

\def\BCT{\textup{\textsf{BCT}}}

\def\open{\textup{\textsf{open}}}

\def\KL{\textup{\textsf{KL}}}
\def\BOOT{\textup{\textsf{BOOT}}}

\def\IND{\textup{\textsf{IND}}}
\def\NFP{\textup{\textsf{NFP}}}

\def\HBU{\textup{\textsf{HBU}}}
\def\HBT{\textup{\textsf{HBT}}}

\def\WHBU{\textup{\textsf{WHBU}}}

\def\mod{\textup{\textsf{mod}}}

\def\fin{\textup{\textsf{fin}}}

\def\eps{\varepsilon}

\def\X{\textup{\textsf{X}}}

\def\CC{\textup{\textsf{CC}}}

\def\ADS{\textup{\textsf{ADS}}}

\newcommand{\rinf}{\rightarrow \infty}

\usepackage{mathrsfs}  

\usepackage{graphicx}
\usepackage{tikz}
\usetikzlibrary{matrix, shapes.misc}
\usepackage{comment,tikz-cd}

\numberwithin{equation}{section}
\numberwithin{thm}{section}

\begin{document}
\title{Numerical choice, Riemann integration, and Reverse Mathematics}
\author{Sam Sanders}
\address{Department of Philosophy II, RUB Bochum, Germany}
\email{sasander@me.com}
\author{Dmytro Taranovsky}
\email{dmytro@alum.mit.edu}

\keywords{numerical choice, Riemann integration, Reverse Mathematics}
\subjclass[2020]{Primary: 03B30, 03F35}

\begin{abstract}
Riemann integration remains a well-known part of mathematics for both historical and conceptual reasons.  
We study basic properties like boundedness of Riemann integrable functions and related classes in mathematical logic.  
On one hand, weak logical systems already establish that a Riemann integrable function on the unit interval is bounded or dominated by a continuous function.  
On the other hand, the following slight generalisation already implies a rather strong logical system, namely the `Big Five' system $\ATR_{0}$ which accommodates \emph{transfinite recursion}.
\begin{center}
\emph{For $f:\R\di \R$ Riemann integrable on any interval $[-a, a]$ for $a>0$, there is continuous $g:\R\di \R$ with $f(x)\leq g(x)$ for all $x\in \R$.}
\end{center}
As part of the \emph{Reverse Mathematics} program, we obtain equivalences for the centred statement and variations involving the axiom of numerical choice.
A central result is that numerical choice for $\Pi_{1}^{1}$-formulas is equivalent to $\ATR_{0}$.  
We also obtain equivalences involving basic properties of metric spaces and establish connections to Kohlenbach's generalisations of weak K\"onig's lemma, Cousin's lemma, and the representation of open sets.    
\end{abstract}
%
\maketitle              

\section{Introduction}
\subsection{Aim and motivation}\label{intro}
The Riemann integral boasts an impressive historical role and remains an important step in analysis for many high-school and undergraduate students.    
In this paper, we study Riemann integration from the point of view of mathematical logic; the following basic statement is central.
\be\label{diff}\tag{\textsf{RB}}
\begin{array}{c}\textup{\emph{For $f:\R\di \R$ Riemann integrable on any interval $[-a, a]$ for $a>0$,}} \\ \textup{\emph{there is continuous $g:\R\di \R$ with $f(x)\leq g(x)$ for all $x\in \R$.}} \end{array}
\ee
In particular, given the central role of the Riemann integral, the juxtaposition between items (a)-(b) is of general interest to any scientist, and of particular interest to anyone interested
in the logical and computational properties of mathematics. 
\begin{itemize}
\item[(a)] Weak logical systems can prove that a Riemann integrable function on $[0,1]$ is bounded or dominated by a continuous function on $[0,1]$.
\item[(b)] Rather strong logical systems are needed to prove the statement \eqref{diff}.  
\end{itemize}
We also study variations of (a)-(b) for other function classes, like sub-continuity, and basic properties of metric spaces, with similar results. 

\smallskip

A more detailed study of these items proceeds as follows: the  program \emph{Reverse Mathematics} (often abbreviated RM; see \cite{simpson2, damurm, kohlenbach2, stillebron} for an introduction) seeks to identify the minimal axioms 
needed to prove theorems of ordinary mathematics.   As a contribution to RM, we obtain equivalences involving the statement \eqref{diff} and variations, and the axiom of numerical choice.  
A central result is our equivalence between the fragment of numerical choice for $\Pi_{1}^{1}$-formulas and the well-known and relatively strong system $\ATR_{0}$, which accommodates transfinite recursion. 

\smallskip

We discuss some preliminaries and definitions in Section \ref{prelim} while our main results are in Sections \ref{cabout} and \ref{duko}.  
In more detail, we shall establish the following.  
\begin{itemize}
\item The system $\ATR_{0}$ is equivalent to various versions of (second-order) numerical choice for $\Pi_{1}^{1}$-formulas (Theorem \ref{hofff}).  
\item In Kohlenbach's higher-order RM, the associated higher-order numerical choice axioms yield a conservative extension of $\ATR_{0}$ (Theorem \ref{hunterX}).  
\item In Kohlenbach's higher-order RM, weak numerical choice is equivalent to \eqref{diff} and variations (Theorem \ref{Taranie}), as well as basic properties of metric spaces (Theorem \ref{tomore}).
\item The previous item yields equivalences for $\ATR_{0}$ and \eqref{diff} restricted to effectively Baire 2 functions (Theorem \ref{coredesigner}).    
\item Higher-order numerical choice is connected to Cousin's lemma (\cite{cousin1}) and Kohlenbach's generalisation of weak K\"onig's lemma (\cite{kohlenbach4}), and the coding of open sets (\cite{dagsamVII}), as discussed in Section \ref{duko}.
\end{itemize}
The penultimate item is interesting as it connects second- and higher-order RM. 
Finally, we provide a graphical overview with some foundational discussion in Section~\ref{grafo}, while suggestions for future research can be found in Section \ref{dasfutur}.

\subsection{Preliminaries and definitions}\label{prelim}
\subsubsection{On higher-order arithmetic}\label{ohoam}
We assume familiarity with Kohlenbach's higher-order RM, the base theory $\RCAo$ in particular. 
The original text is \cite{kohlenbach2} while more recent introductions are in \cites{dagsamXIV, sammetric}. 
A monograph on this topic is forthcoming (\cite{samBOOK}).

\smallskip

Zeroth of all, higher-order RM is officially formulated using types rather than sets.  
Hence, both `$n\in \N$' and `$n^{0}$' are used for `$n$ is a natural number'.  
Similarly, elements of Baire space are identified as `$f^{1}$' or `$f\in \N^{\N}$', while mappings from Baire space to Baire space are denoted $\Phi^{1\di 1}$ or $\Phi:\N^{\N}\di \N^{\N}$.  
We will not use objects of higher rank in this paper and the use of types is kept to a minimum.  We do introduce (standard) sequence notation in the below (Notation \ref{taxio}).

\smallskip

First of all, real numbers and real equality `$=_{\R}$' are defined in $\RCAo$ in the same way as in second-order RM, i.e.\ as fast-converging Cauchy sequences.  
A function from reals to reals, denoted $F:\R\di \R$, is then given by $\Phi:\N^{\N}\di \N^{\N}$ such that 
\[
(\forall x, y\in \R)(x=_{\R}y\di \Phi(x)=_{\R} \Phi(y)).
\]
Secondly, we consider the following essential axiom where the functional $E$ is also called \emph{Kleene's quantifier $\exists^{2}$} and is discontinuous on Baire space:
\be\tag{$\exists^{2}$}
(\exists E:\N^{\N}\di \{0,1\})(\forall f\in\N^{\N})( E(f)=0 \asa (\exists n\in \N)(f(n)=0)).
\ee
We write $\ACAo\equiv \RCAo+(\exists^{2})$ and observe that the latter proves the same second-order sentences as $\ACA_{0}$ (see \cite{hunterphd}).
We shall mostly work in $\ACAo$, which is convenient as set of reals are then given via characteristic functions, as follows.  
\bdefi[Sets]\label{char}~
\begin{enumerate}
\renewcommand{\theenumi}{\alph{enumi}}
\item A set of reals $A$ is given by a function $F_{A}:\R\di \{0,1\}$; we write $x\in A$ for $ F_{A}(x)=1$, for any $x\in \R$.\label{tinkzzz}
\item We write `$A\subseteq B$' if we have $(\forall x\in \R)(x\in A\di x\in B)$.  
\item A set $O\subseteq \R$ is \emph{open} in case $x\in O$ implies that there is $k\in \N$ such that $B(x, \frac{1}{2^{k}})\subseteq O$.\label{qzopen}  
\item The complement $O^{c}$ of an open set $O\subset \R$ is called closed.  
\item A set $O\subseteq \R$ is \emph{RM-open} if there are sequences $(a_{n})_{n\in \N}, (b_{n})_{n\in \N}$ of reals such that $x\in O$ if and only if $x\in \cup_{n\in \N}(a_{n}, b_{n})$ for all $x\in \R$ (\cite{simpson2}*{II.5.6}).\label{daz}
\item The complement $O^{c}$ of an RM-open set $O\subset \R$ is called RM-closed.  
\item A set $A\subset \R$ has \emph{measure zero} if for any $\eps>0$ there is a sequence of open intervals $(I_{n})_{n\in \N}$ such that $\cup_{n\in \N}I_{n}$ covers $A$ and $\eps>\sum_{n=0}^{\infty}|I_{n}|$. 
\item A set $A\subset \R$ is \emph{enumerable} if there is a sequence $(x_{n})_{n\in \N}$ that includes all elements of $A$. 
\item A set $A\subset \R$ is \emph{countable} if there is $Y:\R\di \N$ satisfying
\[
(\forall x, y\in A)(Y(x)=_{\N}Y(y)\di x=_{\R}y),
\]
where $Y$ is called `injective on $A$'.  
\end{enumerate}
\edefi
With this convention in place, we can adopt the usual definitions of nowhere dense set, $\bf{F}_{\sigma}$-set, the Baire property, et cetera from the literature. 
We do feel the need to recall some standard sequence notation from type theory.
\begin{nota}\label{taxio}\rm
Finite sequences of naturals are denoted $\sigma\in \N^{<\N}$ or $\sigma^{0^{*}}$ with the empty sequence being `$\langle \rangle$'.  
We do not always distinguish between a natural number $n$ and the associated one-element sequence $\langle n\rangle$. 
For $\sigma^{0^{*}}$ and $\tau^{0^{*}}$, the `concatenation of $\sigma$ followed by $\tau$' is denoted $\sigma*\tau$. 
For $\sigma^{0^{*}}=\langle n_{0}, \dots n_{k}\rangle$, the `length $|\sigma|$' is $k+1$ while $\overline{\sigma}m$ for $m<|\sigma|$ is $\sigma$ `cut off' after the first $m+1$ elements.  
We assume a standard pairing function is given, which codes finite sequences as natural numbers.  
We use similar notations for infinite sequences of naturals with their obvious meaning. 
\end{nota}
Thirdly, following Notation \ref{taxio}, consider the following axiom where the functional $\SS^{2}$ is often called \emph{the Suslin functional} (\cite{kohlenbach2, avi2, yamayamaharehare}):
\be\tag{$\SS^{2}$}
(\exists\SS:\N^{\N}\di \{0,1\})(\forall f \in \N^{\N})\big[  (\exists g \in \N^{\N})(\forall n \in \N)(f(\overline{g}n)=0)\asa \SS(f)=0  \big].
\ee
By definition, the Suslin functional $\SS^{2}$ can decide whether a $\Sigma_{1}^{1}$-formula in normal form, i.e.\ as in the left-hand side of $(\SS^{2})$, is true or false.   
The system $\FIVE^{\omega}\equiv \RCAo+(\SS^{2})$ proves the same second-order sentences as $\FIVE$ via a straightforward variation of the conservation results from \cite{hunterphd}.  

\smallskip

We also consider the functional $\SS_{k}^{2}$ which decides the truth or falsity of $\Sigma_{k}^{1}$-formulas in normal form; we define 
the system $\SIXK$ as $\RCAo+(\SS_{k}^{2})$, where  $(\SS_{k}^{2})$ expresses that $\SS_{k}^{2}$ exists.  
We define $\Z_{2}^{\omega}$ as $\cup_{k}\SIXK$ as one possible higher-order version of $\Z_{2}$.
The functionals $\nu_{n}$ from \cite{boekskeopendoen}*{p.\ 129} are essentially $\SS_{n}^{2}$ strengthened to return a witness (if existent) to the $\Sigma_{n}^{1}$-formula at hand.  
The operator $\nu_{n}$ is Hilbert-Bernays' $\nu$ from \cite{hillebilly2}*{p.\ 479} restricted to $\Sigma_{n}^{1}$-formulas. 

\smallskip

\noindent
Fourth, we introduce Kleene's quantifier $\exists^{3}$ as follows:
\be\tag{$\exists^{3}$}
(\exists E: (\N^{\N}\di \N)\di \N)(\forall Y:\N^{\N}\di \N)\big[  (\exists f \in \N^{\N})(Y(f)=0)\asa E(Y)=0  \big].
\ee
Both $\Z_{2}^{\Omega}\equiv \RCAo+(\exists^{3})$ and $\Z_{2}^\omega\equiv \cup_{k}\SIXK$ are conservative over $\Z_{2}$ as shown in \cite{hunterphd}.
The functional from $(\exists^{3})$ is also called `$\exists^{3}$', and we use the same convention for other functionals.  Hilbert-Bernays' operator $\nu$ from \cite{hillebilly2}*{p.\ 479} is essentially Kleene's $\exists^{3}$, modulo a non-trivial fragment of the Axiom of (quantifier-free) Choice.    

\smallskip

Fifth, like its second-order counterpart, the development of higher-order RM sometimes takes place over an \emph{extension} of the base theory.  
The following fragment of countable choice, not provable\footnote{One readily proves that $\QFAC^{0,1}$ is equivalent to $\CC(\R)$, i.e.\ countable choice for countable unions of sets of reals, over $\ZF$ (\cite{heerlijkheid}).} in $\ZF$, plays an important role.
\begin{princ}[$\QFAC^{0,1}$]
Let $\varphi$ be quantifier-free with $(\forall n\in \N)(\exists f\in \N^{\N})\varphi(f, n)$, then there exists a sequence $  (f_{n})_{n\in \N}$ in $\N^{\N}$ with $(\forall n\in \N)\varphi(f_{n}, n)$.
\end{princ}
The local equivalence between sequential and `epsilon-delta' continuity cannot be proved in $\ZF$ (\cite{heerlijkheid}), but can be established in $\RCAo+\QFAC^{0,1}$, as first shown by Kohlenbach in \cite{kohlenbach2}.  
In this light, it should not be a surprise that $\RCAo+\QFAC^{0,1}$ often occurs as a base theory.  Similarly, extra induction axioms are sometimes used in second-order RM (\cite{neeman}) and we will often use the following fragment of induction, where formulas of the same form as $\varphi(n)$ are called\footnote{The usual formula hierarchy based on the classes $\Sigma_{k}^{i}$ for $i=0, 1$ only allows for first- and second-order parameters.  
As a generalisation of $\Sigma_{1}^{1}$ to higher-order parameters, we let a \emph{$\Sigma$-formula} be any formula of the form $(\exists f\in 2^{\N})(Y(f, n)=0)$ for third-order $Y$. } a \emph{$\Sigma$-formula}; the negation of the latter is called a \emph{$\Pi$-formula}. 
\begin{princ}[$\SIND$]  
The induction axiom for formulas of the form $\varphi(n)\equiv (\exists f\in \N^{\N})(Y(f, n)=0)$ for any $Y^{2}$.   
\end{princ}
Finally, it is an empirical observation that, on one hand, many third-order theorems about (possibly) discontinuous functions are equivalent to the Big Five of RM (\cite{dagsamXIV}), 
while on the other hand many third-order theorems are provable in $\Z_{2}^{\Omega}$ but not in $\Z_{2}^{\omega}+\QFAC^{0,1}$.  
The weakest such `natural' theorem is well-known as the \emph{uncountability of the reals}, formulated as follows:
\begin{center} 
$\NIN_{[0,1]}$: there is no injection from $[0,1]$ to $\N$.  
\end{center}
A long list of theorems that imply $\NIN_{[0,1]}$ can be found in \cites{dagsamX, samBOOK}.  A surprisingly strong and natural principle not provable in $\Z_{2}^{\omega}+\QFAC^{0,1}$ is the coding principle $\open$, studied in \cite{dagsamVII, samBOOK}, which expresses that an open set as in Definition \ref{char} is RM-open.  
The principle $\open$ is equivalent to e.g.\ the supremum principle for semi-continuous functions, the Urysohn lemma, and the Tietze extension theorem, all formulated in third-order arithmetic.  

\subsubsection{Definitions and basic results}
We use the standard definition of the Riemann integral.  
A modulus is well-known in RM and related areas (see e.g.\ \cite{sayo, bish1, aberth}). 
\begin{defi}[Riemann integral] Let $a<_{\R}b$ be given.  
\begin{itemize}
\item A finite sequence $P:=(t_{0}, I_{0}, \dots, t_{k}, I_{k})$ is a \emph{tagged partition} of $[a,b]$ if the `tag' $t_{i}\in \R$ is in the interval $ I_{i}$ for $i\leq k$, and the $I_{i}$ partition $I$.
\item The \emph{mesh} $\|P\|$ of a tagged partition $P$ is $\max_{1\leq i\leq k} | I_{i}|$.
\item For a tagged partition $P$ and $f:[a,b]\di \R$, the \emph{Riemann sum} is $S(f, P, a, b):=\sum_{i=1}^{k} f(t_{i-1})\cdot |I_{i}|$.
\item A function $f:[a,b]\di \R$ is \emph{Riemann integrable} on $[a, b]$ if for any $k\in \N$ there is $N\in \N$ such that for tagged partitions $P, Q$ of $[a,b]$ with $\|P\|, \|Q\|\leq \frac{1}{2^{N}}$, we have $|S(f, P, a, b)-S(f, Q, a, b)|<\frac{1}{2^{k}}$.
\item A function $f:[a,b]\di \R$ is \emph{effectively} Riemann integrable if there is $g\in \N^{\N}$, called a modulus, such that we can take $N=g(k)$ for all $k\in \N$ in the previous item. 
\item A function $f:\R\di \R$ is Riemann integrable \emph{everywhere} if it is Riemann integrable on each real interval $[a, b]\subset \R$. 
\item A function $f:\R\di \R$ is \emph{effectively} Riemann integrable everywhere if there is $g:( \R^{2}\times \N)\di \N$ such that $g(k,a, b)$ is a modulus for all $a, b\in \R$.
\end{itemize}
\edefi
The Vitali-Lebesgue theorem states that Riemann integrability is equivalent to being continuous almost everywhere and bounded on every interval.
Now, a function $f:\R\di \R$ is usually called continuous almost everywhere (ae) if the set of discontinuity points $D_{f}$ has measure zero.    
However, the set $D_{f}$ need not exist in weak logical systems as its definition involves quantifiers over $\R$.
The following definition avoids introducing the set $D_{f}$ in the same way the usual definition of `measure zero set' avoids introducing the Lebesgue measure.  
\bdefi
A function $f:\R\di \R$ is \emph{continuous ae} in case for any $\eps>0$, there is a sequence of intervals $(I_{n})_{n\in \N}$ of total length at most $\eps$ such that if $f$ is discontinuous at $x\in \R$, we have $x\in \cup_{n\in \N}I_{n}$.  
\edefi
\bdefi\label{zepsco}
A function $f:\R\di \R$ is \emph{$\eps$-almost continuous ae} if for any $\eps'>0$, there is a sequence of intervals $(I_{n})_{n\in \N}$ of total length at most $\eps'$ such that for $x\not \in \cup_{n\in \N}I_{n}$, 
we have that $\limsup_{y\di x} |f(x)-f(y)|\leq \eps$.
\edefi
One direction of the Vitali-Lebesgue theorem is then readily proved.  
\begin{thm}[$\RCAo+\WKL+\QFAC^{0,1}$]\label{bengas2}
Let $f:[0,1]\di \R$ be continuous ae and bounded.
Then $f$ is Riemann integrable on $[0,1]$.
\end{thm}
\begin{proof}
In case $f$ is continuous, the theorem follows by \cite{dagsamXIV}*{\S2}.  In case $f$ is discontinuous, we have access to $(\exists^{2})$ by \cite{kohlenbach2}*{Prop.\ 3.14}.  
Suppose $|f(x)|\leq M\in \N$ for all $x\in [0,1]$ and fix $\eps>0$.  Since $f$ is continuous ae, let $(I_{n})_{n\in \N}$ be the associated sequence of open intervals with total length $\eps/8M$.   
Define the closed set $C$ as $[0,1]\setminus \cup_{n\in \N}I_{n}$.  Now assume that $f$ is uniformly continuous on $C$ (relative to $\R$), 
i.e.\ there is $\delta>0$ such that for $x\in C, y\in [0,1]$, we have $|x-y|<\delta\di |f(x)-f(y)|<\eps/8$.    
Consider tagged partitions $P, Q$ with mesh at most $\delta$ and verify that $|S(f, P)-S(f, Q)|<\eps$ using the common refinement of $P$ and $Q$, i.e.\ $f$ is Riemann integrable as required.

\smallskip

Finally, to show that $f$ is uniformly continuous on $C$ (relative to $\R$), suppose not, i.e.\ 
there is $k_{0}\in \N$ such that for all $N\in \N$, there are $x\in C$ and $y\in B(x, \frac{1}{2^{N}})$ such that $|f(x)-f(y)|\geq \frac{1}{2^{k_{0}}}$.
Apply $\QFAC^{0,1}$ to obtain sequences $(x_{N})_{N\in \N}$ and $(y_{N})_{N\in \N}$ such that for all $N\in \N$, $x_{N}\in C$, $y_{N}\in B(x_{N}, \frac{1}{2^{N}})$, and $|f(x_{N})-f(y_{N})|\geq \frac{1}{2^{k_{0}}}$.
By sequential compactness, $(x_{N})_{N\in \N}$ has a convergent sub-sequence $(z_{N})_{N\in \N}$, say with limit $z\in [0,1]$.  
We let $(w_{N})_{N\in \N}$ be the associated sub-sequence of $(y_{N})_{N\in \N}$.  
Since $C$ is closed, we have $z\in C$ but also that $|f(z)-f(w_{N})|>\frac{1}{2^{k_{0}}}$ for arbitrary large $N\in \N$, a contradiction as $f$ is continuous at all reals in $ C$. 
\end{proof}
One readily shows that the other direction of the Vitali-Lebesgue theorem, i.e.\ Riemann integrability implies continuity ae, implies $\NIN_{[0,1]}$ from Section \ref{ohoam} and is therefore not provable from the Big Five, or $\Z_{2}^{\omega}+\QFAC^{0,1}$.  The RM of the Vitali-Lebesgue theorem is connected to Tao's pigeon hole principle for measure (\cite{taomes}) and explored in \cite{samBIG2}.  
We also have the following result; Patrick Uftring has provided a proof\footnote{Apply the definition of Riemann integrability for $k=0$ and let $N_{0}\in \N$ be the resulting number.  
Let $M_{0}$ be an upper bound for $f(t_{i})$ and $f(x_{i})$ where the $x_{i}$ and $t_{i}$ come from an equidistant partition with mesh $\frac{1}{2^{N_{0}}}$.
Now verify that $f$ is bounded on $[0,1]$ by $2^{N_{0}}+M_{0}$.\label{uftring}} that Riemann integrable functions are bounded in $\RCAo$.  
\begin{thm}[$\RCAo+\WKL_{0}+\QFAC^{0,1}+\SIND$]\label{flalala}
A function is Riemann integrable iff it is bounded and for any $\eps>0$, it is $\eps$-almost continuous ae.  
\end{thm}
\begin{proof}
The proof of Theorem \ref{bengas2} establishes the reverse direction.  For the forward direction, 
given $\eps>0$, take a uniform mesh $\delta>0$ partition that guarantees that Riemann sums differ by at most $\eps$.  
For the partition, using $\SIND$, choose length $1/\delta$ sequences $x_i$ and $y_i$ that maximise the possible differences in Riemann sums for the function $f$, up to some $\eps\delta$ additive factor.  
For almost all indices $i$ (as $\eps\di 0$), we have  $|f(x_i)-f(y_i)|<\sqrt \eps$, since $\eps$ is $o(\sqrt \eps)$.  
Furthermore, as the choice of $x_{i}$ and $y_{i}$ essentially maximises possible differences, in the corresponding intervals, $f$ is $(\sqrt{\eps}+\eps)$-almost continuous, as desired.
\end{proof}
Note that $\SIND$ is not used for the reverse direction while $\QFAC^{0,1}$ is not used for the forward direction.
Remark \ref{trixie} discusses the role of the Axiom of Choice in Theorem \ref{flalala}.  Below, we make use of similar generalisations of $\WKL$ from \cite{kohlenbach4}*{\S5}.  
\begin{rem}\label{trixie}\rm
Since $\QFAC^{0,1}$ is a fragment of the Axiom of Choice not provable in $\ZF$, it is interesting to observe that the presence of $\SIND$ allows us to weaken $\QFAC^{0,1}$ to the following generalisation of $\WWKL$ (see \cite{simpson2}*{X} for the latter).
\begin{center}
\emph{Every infinite and positive measure $0$-$1$-tree $T$ given by a $\Sigma$-formula, has an infinite path.}
\end{center}
Here the path -but not necessarily the tree- exists as a set.  Positive measure means that for some $\eps>0$, for each $n\in \N$, the tree has at least $\eps 2^n$ nodes at depth $n$; this makes sense using $\SIND$.
 For the proof, the tree is formed by recursively even partitioning of the interval, keeping as branches sub-intervals on which $f$ varies by more than $\eps$ and that are not excluded by the chosen $²\eps$-length interval sequence from Definition \ref{zepsco}.  If for every $\eps$ (the tree depends on $\eps$), the tree is not of positive measure, we get Riemann integrability, while a path contradicts the continuity.
\end{rem}
Finally, we list the remaining definitions of the function classes to be studied.  
\bdefi\label{flung} 
For $f:\R\di \R$, we have the following definitions:
\begin{itemize}
\item $f$ is \emph{quasi-continuous} at $x_{0}\in\R$ if for any $k,N\in \N$, there is an open interval ${ (a, b)\subset B(x_{0}, \frac{1}{2^{N}})}$ with $(\forall x\in (a, b)) (|f(x_{0})-f(x)|<\frac{1}{2^{k}})$,
\item $f$ is \emph{locally bounded} at $x_{0}\in \R$ if $(\exists N\in \N)(\forall y\in B(x_{0}, \frac{1}{2^{N}}))(|f(y)|\leq N )$,
\item $f$ is \emph{sub-continuous} at $x_{0}\in \R$ if for any $(y_{n})_{n\in \N}$ with limit $x_{0}$, the sequence of reals $(f(y_{n}))_{n\in \N}$ has a convergent sub-sequence,
\item $f$ is \emph{measurable} if for any open $V\subset \R$, $f^{-1}(V)$ is the union of an $\bf F_{\sigma}$ and a measure zero set \(see e.g.\ \cite{taomes}*{Ex.\ 1.2.19 and Lem.\ 1.3.9}\),
\item $f$ is \emph{Baire measurable} \(\cite{hijanis}\) if for every open $V\subset \R$, the set $f^{-1}(V)$ has the Baire property, 
\item $f$ is \emph{simply continuous} \(sico\) if for any open $V\subset \R$, the set $f^{-1}(V)$ is the union of an open and a nowhere dense set \(\cites{morebors, bronnaz, biwas}\).
\end{itemize}
\edefi
Many more function classes are studied in \cite{samBOOK} and we would be interested in other classes that behave in the same way as the Riemann integrable functions.

\section{About and around numerical choice}\label{cabout}
\subsection{Introduction}
In the below, we study the logical properties of the numerical choice axiom for $\Pi_{1}^{1}$-formulas as follows. 
\begin{princ}[$\BNC$]\label{zefdes}
For any formula $\varphi \in \Pi_{1}^{1}$, we have  
\[
(\forall n\in \N)(\exists m\in \N)\varphi(n, m)\di (\exists g\in \N^{\N})(\forall n\in \N)(\exists m\leq g(n))\varphi(n, m).
\]
\end{princ}
\begin{princ}[$\ENC$]\label{zefdes2}
For any formula $\varphi \in \Pi_{1}^{1}$, we have 
\[
(\forall n\in \N)(\exists m\in \N)\varphi(n, m)\di (\exists g\in \N^{\N})(\forall n\in \N)\varphi(n, g(n)).
\]
\end{princ}
The acronyms $\BNC$ and $\ENC$ stand for `bounding' and `exact' numerical choice, in accordance with the properties of the associated choice functions.  
We also wish to study higher-order numerical choice; we let $\ENC^{\omega}$ and $\BNC^{\omega}$ be the associated second-order principles with higher-order parameters allowed in the formula $\varphi$.  

\smallskip

We shall prove the following results in the below sections. 
\begin{itemize}
\item Over $\RCA_{0}$, we have $\ATR_{0}\asa \ENC\asa \BNC$ (Theorem \ref{hofff}).
\item $\ACAo+\ENC^{\omega}+\QFAC^{0,1}$ is conservative over $\ATR_{0}$ (Theorem \ref{hunterX}).  
\item Both $\ENC^{\omega}$ and $\BNC^{\omega}$ cannot be proved in strong systems like $\Z_{2}^{\omega}+\QFAC^{0,1}$; both yield stronger comprehension like $\SIX$ when combined with higher-order comprehension like $\FIVE^{\omega}$ (Theorem~\ref{slif}). 
\item A number of natural equivalences for $\BNC^{\omega}$ exist, involving real analysis (Theorem \ref{Taranie}) and metric spaces (Theorem \ref{tomore}), including the domination of Riemann integrable functions as in \eqref{diff} from Section \ref{intro}.  
\item We connect numerical choice to Cousin's lemma (\cite{cousin1}) and Kohlenbach's generalisation of weak K\"onig's lemma $\Phi_{1}$-$\WKL$ (\cite{kohlenbach4}) in Sections \ref{notenoughleeb}-\ref{kokolema}. 
A pleasing equivalence is $\ENC^{\omega}\asa [\BNC^{\omega}+\Phi_{1}\textup{-}\WKL]$ over extra induction.  
\item We show that numerical choice is intimately related to the coding principle $\open$ from \cite{dagsamVII}, which simply expresses that open sets have RM-codes.  
\end{itemize}
We believe there to be many more potential equivalences for numerical choice. 

\subsection{Second-order equivalences and around}\label{enc1}
We show that the numerical choice principles $\BNC$ and $\ENC$ are equivalent to $\ATR_{0}$ over $\RCA_{0}$.  
We obtain a conservation result for the associated higher-order system $\ENC^{\omega}$ and $\ATR_{0}$.  

\smallskip

Firstly, we have the following second-order equivalences.   
The complexity in the final item is optimal as $\ACA_0$ proves that for any function $f$ that is the lim inf of a sequence of continuous functions, there is $g:\Q^2\di \R\cup\{-\infty,+\infty\}$ that lists the supremum of $f$ on every rational interval.
\begin{thm}[$\RCA_{0}$]\label{hofff} The following are equivalent. 
\begin{enumerate}
\renewcommand{\theenumi}{\alph{enumi}}
\item The principle $\ATR_{0}$.
\item The separation principle for $\Sigma_{1}^{1}$-formulas \(\cite{simpson2}*{V.5.1}\).
\item The principle $\BNC$ as in Principle \ref{zefdes}.
\item The principle $\ENC$ as in Principle \ref{zefdes2}.
\item Every $\R\di \R$-function that is the lim sup of a sequence of continuous functions and is non-zero for at most one argument on each unit interval, is dominated by a continuous function.\label{kuhl}
\end{enumerate}
\end{thm}
\begin{proof}
First of all, the restriction of Principle \ref{zefdes2} to $m\in \{0,1\}$ yields the separation principle for $\Sigma_{1}^{1}$-formulas, which is equivalent to $\ATR_0$ by \cite{simpson2}*{V.5.1}.  
To see that the full principle is provable in $\ATR_{0}$, let $\varphi(n,m)$ have the form $(\forall X\subset \N)\phi(X, n,m)$ and let $Z$ denote the parameters of the arithmetical formula $\phi$ not shown. 
For each $\Pi_{1}^{1}(Z)$ statement, assign an ordinal by converting the statement to well-foundedness of a $Z$-recursive order and taking its ordinal, using $\infty$ for false statements. 

\smallskip

Let $\alpha_{n,m}$ be the ordinal for $(\forall X\subset \N)\phi(X, n,m)$ and define $\alpha_{n}:=\min_{m}\alpha(n,m)$ and $\alpha:=\sup_{n}\alpha_{n}$.  
Then $\alpha$ exists because we can concatenate $X$-recursive well-orderings exceeding $\alpha_{n}$ and get an upper bound on $\alpha$.  
Note that $\alpha_{n}$ can be defined uniformly as we can convert $(\exists n\in \N)\varphi(n,m)$ to a $\Pi^1_1$-formula in $\ATR_{0}$; the conversion must be chosen such that the resulting ordinal is not below $\alpha_{n}$. 

\smallskip

Now, a choice function for $(\forall n\in \N)(\exists m \in \N)\varphi(n,m)$ exists in $L_{\alpha+1}(Z)$ and is hyperarithmetical in $Z$.  
This is because $L_{\alpha+1}(Z)$ contains the set of all $\Pi_{1}^{1}(Z)$-statements whose ordinals are at most $\alpha$, along with their assigned ordinals.
Hence, $\ATR_{0}$ implies Principle \ref{zefdes2}, as required.  

\smallskip

Secondly, we can derive $\ACA_0$ from Principle \ref{zefdes} by using it (for every $Z$) to give bounds for true $\Sigma^0_1(Z)$-statements, and thus computing the Turing jump of $Z$.  
Next, let '$\prec$' be a $Z$-recursive well-ordering.  To obtain $\ATR_{0}$, it suffices to show that, starting at $Z$, the Turing jump can be iterated along $\prec$.  
To this end, let $\phi(X,n,m)$ be the formula 
\begin{center}
\emph{decode $n$ as $(a,e)$; if $X$ encodes an iteration of the Turing jump \(starting at $Z$\) for the restriction of $\prec$ below $a$, and the Turing machine $e$ with oracle $(Z,X)$ halts, then it halts in at most $ m$ steps.}
\end{center}
We have $(\forall n\in \N)(\exists m\in \N)(\forall X\subset\N)  \phi(X,n,m)$) because otherwise we would get conflicting arithmetic transfinite recursions for an initial segment of $\prec$, which we can then compare (in $\ACA_0$) to get an infinite $\prec$-descending sequence.   Observe that item \eqref{kuhl} yields $\ACA_{0}$ via a similar proof.  

\smallskip

Now let $g$ be as in Principle \ref{zefdes}.  Then the iterated Turing jump is computable from $Z,g$ by converting transfinite recursion into bounded recursion using $g$.  Indeed, by $\WKL_{0}$, a failure of the bounded recursion to terminate would give an infinite $\prec$-descending path.  Next, if there are no inconsistencies below $a$, then the Turing jump can be iterated from $Z$ along $\prec$ up to $a$ (uniquely, per above), and thus no inconsistencies occur at $a$ either.  
Thus, each inconsistency gives a $\prec$-lower inconsistency, and since an inconsistency is an arithmetical property, an infinite $\prec$-descending sequence.

\smallskip

For the final item (a special case of $\BNC$), note that $X$ for the above is unique, and the required tests are $\Pi^0_2(Z)$, and convert to using reals as in the proof of Theorem~\ref{Taranie} below.  Here, $\Pi^0_2$ corresponds to $\lim \sup$ of continuous functions.  Observe that $\ACA_0$ is derivable as in the previous paragraphs.
\end{proof}
Secondly, we have the following conservation result. 
\begin{thm}\label{hunterX}~
 The system $\ACAo+\ENC^{\omega}+\QFAC^{0,1}$ is conservative over $\ATR_{0}$.
\end{thm}
\begin{proof}
Hunter proves in \cite{hunterphd}*{Theorem 2.5 and Corollary 2.6-2.7} that $\ACAo$, $\ACAo+\QFAC^{0,1}$, and $\Z_{2}^{\Omega}$ are conservative over respectively $\ACA_{0}$, $\SAC$, and $\Z_{2}$.  
We will sketch Hunter's proof and show it is readily adapted to $\ACAo+\ENC^{\omega}+\QFAC^{0,1}$.  
We should mention that $\ATR_{0}$ proves $\SAC$ and the latter is conservative over \textsf{PA}; this is not true in the presence of full induction.  
The minimal $\omega$-model of $\SAC$ consists of all hyperarithmetical sets.  All these facts may be found in \cite{simpson2}.  

\smallskip

First of all, the main part of Hunter's proof consists of the following (lengthy, hence omitted) construction, where we note that $\ACA_0$ is formalised using function rather than set variables, a mere technicality.  
\begin{center}
\emph{Let $\M$ be a \(second-order\) model of $\ACA_0$; there exists a term model $\mathcal{N}$ of $\ACAo$ whose second-order part, i.e.\ the collection of all type-$0$ and type-$1$ elements, is isomorphic to~$\M$.}
\end{center}
This construction then immediately yields that $\ACAo$ is conservative over $\ACA_{0}$.  
Indeed, if a second-order sentence $\Phi$ is provable from $\ACAo$ then it holds in all models of $\ACAo$, so it holds in the associated term model $\mathcal{N}$. 
Since $\Phi$ involves only second-order objects, it holds in $\M$ by isomorphism; since~$\M$ was arbitrary, $\Phi$ holds in all models of $\ACA_0$ and hence is provable from $\ACA_0$ by completeness.

\smallskip

Secondly, $\ACAo+\QFAC^{0,1}$ implies $\SAC$, since the former is simply a higher-order generalisation of the latter. 
The theory $\SAC$ also includes $\ACA_{0}$, which is implied by $(\exists^{2})$. 
Now let $\M$ be a model of $\SAC$; construct a term model $\mathcal{N}$ from $\M$ exactly as in the previous part of the proof. Then $\mathcal{N}$ satisfies $\ACAo$, since $\M$ satisfies $\ACA_0$. It remains only for us to show that $\mathcal{N}$ satisfies $\QFAC^{0,1}$.  

\smallskip

Let $\Phi$ be a quantifier-free formula in the language of $\mathcal{N}$ and suppose for every $n \in \N$ there is a type-$1$ function $f$ such that $\Phi(n, f)$ holds. We will show that there is a type-$(0 \to 1)$ functional $F$ such that for all $n$, $\Phi(n, F(n))$ holds.  The following observation by Hunter, which we call $(\textsf{O})$, is now crucial.
\begin{center}
\emph{For any assignment to $\Phi$'s other parameters, $\Phi(n, f)$ is equivalent to an arithmetical formula $\Phi'$ in the language of $\M$.} 
\end{center}
Indeed, since $\M$ satisfies $\SAC$, there is in $\M$ a type-$1$ function $g$ such that for all $n$, $\Phi'(n, g_n)$ holds, where $g_n(x) = g(\langle x, n \rangle)$ using the standard pairing function $\langle \cdot, \cdot \rangle$ on $\N$.   Define $F$ by the rule:
$
F(n) (x) = g_n(x) = g(\langle x, n \rangle ),
$
and we are done.

\smallskip

Finally, to show that $\ACAo+\ENC^{\omega}+\QFAC^{0,1}$ is conservative over $\ATR_{0}$, we observe that $\ENC\asa \ATR_{0}$ by Theorem \ref{hofff}, while 
 $\ENC^{\omega}$ reduces to $\ENC$ in Hunter's term model $\mathcal{N}$ by the previous centred observation $(\textsf{O})$. 
\end{proof}
The previous proof can be generalised to other higher-order principles (see \cite{samBOOK}*{Appendix}).  We intend to extend these results in a future publication. 

\smallskip

Thirdly, while $\ENC$ and $\BNC$ are equivalent by Theorem \ref{hofff}, $\ENC^{\omega}$ seems stronger than $\BNC^{\omega}$.  Nonetheless, both cannot be proved in relatively strong systems by Theorem \ref{slif}. 
Recall that $\NIN_{[0,1]}$ from \cite{dagsamX} is the statement that there is no injection from $[0,1]$ to $\N$.  
Similarly, the stronger principle $\cocode_{0}$ from \cite{dagsamXI} is the statement that a countable set of reals can be enumerated.  It is known that $\cocode_{0}$ implies $\NIN_{[0,1]}$ and both are unprovable in $\Z_{2}^{\omega}+\QFAC^{0,1}$, but provable in $\Z_{2}^{\Omega}$.
\begin{thm}\label{slif}
We have the following.
\begin{enumerate}
\renewcommand{\theenumi}{\alph{enumi}}
\item The system $\ACAo+\ENC^{\omega}$ proves $\cocode_{0}$. 
\item The system $\ACAo+\ENC^{\omega}$ proves $\ATR_{0}$.  
\item The system $\FIVE^{\omega}+\ENC^{\omega}$ proves $\SIX$.  
\item The system $\ACAo+\BNC^{\omega}$ proves $\NIN_{[0,1]}$.  
\end{enumerate}
\end{thm}
\begin{proof}
Item (a) implies items (b) and (c) by \cite{dagsamX}*{\S3.2.5} and \cite{dagsamXI}*{\S1.3.3}.  To establish item (a), let $A\subset \R$ be such that $Y:\R\di \N$ is injective on $A$.  Now consider
\[\textstyle
(\forall n, k\in \N)(\exists q\in \Q)(\forall x\in A)(Y(x)=n\di  |q-x|<\frac{1}{2^{k}} ).
\]
which has the right syntactic form to apply $\ENC^{\omega}$.  The associate sequence readily yields an enumeration of $A$. 

\smallskip

Finally, to establish item (d), let $(q_{n})_{n\in \N}$ be an enumeration of the rationals in $[0,1]\cap \Q$.  
Now let $Y:[0,1]\di \N$ be injective and consider
\[
(\forall a, b \in \Q\cap [0,1], n\in \N)(\exists m\in \N)( a<_{\R} q_{m}<_{\R} b\wedge (\forall x\in [q_{m}, b))(  Y(x)\ne n)   ).  
\]
Apply $\BNC^{\omega}$ and let $g:(\Q^{2}\times \N)\di \N$ be the associated function.  Define $h(a, b,n)$ as the maximum of those $q_{m}$ with ${m\leq g(a, b, n) \wedge a <q_{m}<b} $.  
Now consider the sequences defined by $a_{0}=0$, $b_{0}=1$ and $a_{n+1}:=h( a_{n}, b_{n}, n ) $ and $b_{n+1}=\frac{a_{n+1}+b_{n}}{2}$, which converge to say $y\in [0,1]$.  
By construction, we have that $Y(y)\ne n$ for all $n\in \N$, which yields a contradiction.  
\end{proof}
Finally, we speculate that $\BNC^{\omega}$ cannot prove that a countable set has measure zero, which is a statement intermediate between $\cocode_{0}$ and $\NIN_{[0,1]}$.

\subsection{Higher-order equivalences}\label{enc2}
We obtain equivalences between numerical choice $\BNC^{\omega}$ and e.g.\ the fact that 
Riemann integrable or sub-continuous functions are bounded above by a continuous function on the reals as in \eqref{diff} (Theorem \ref{Taranie}).  By Theorem~\ref{hofff}, the latter 
statement is surprisingly strong, especially in light of the following slight restriction.  We note that quasi-continuous functions are simply continuous but not vice versa, while both are said to be closely related (\cite{nieuwebron}).    
\begin{thm}[$\RCAo+\WKL_{0}+\QFAC^{0,1}$]\label{weaklame}~
\begin{itemize}
\item A sub-continuous $f:[0,1]\di \R$ is bounded on $[0,1]$. 
\item A Riemann integrable $f:[0,1]\di \R$ is bounded on $[0,1]$. 
\item For a quasi-continuous and Riemann integrable $f:\R\di \R$, there is continuous $g:\R\di \R$ with $f(x)\leq g(x)$ for all $x\in \R$. 
\end{itemize}
If we replace `bounded' by `dominated by a continuous function', the items are provable in $\RCAo+\QFAC^{0,1}$. 
\end{thm}
\begin{proof}
To establish the first two items, in case $f:[0,1]\di \R$ is continuous, it is bounded on $[0,1]$, as proved in \cite{dagsamXIV}*{\S2}.  
In case the former is discontinuous, we obtain $(\exists^{2})$ by \cite{kohlenbach2}*{Prop.\ 3.14}.  
Note that $(\exists^{2})\di \ACA_{0}$, i.e.\ we may use sequential compactness by \cite{simpson2}*{III.2.2}.
Now suppose $f$ is unbounded on $[0,1]$, i.e.\ $(\forall n\in \N)(\exists x\in [0,1])(|f(x)|>n)$.  
Apply $\QFAC^{0,1}$ and let $(x_{n})_{n\in \N}$ be the resulting sequence, i.e.\ satisfying $|f(x_{n})|>n$ for all $n\in \N$.  
By sequential compactness, there is a convergent subsequence, say with limit $y\in [0,1]$.  Clearly, $f$ is not locally bounded at $y\in [0,1]$. 
Moreover, $f$ is not Riemann integrable on $[0,1]$ as Riemann sums can be changed by an arbitrary amount if we choose the right points close enough to $y$. 

\smallskip

For the third item, observe that by definition $(\exists x\in [a,b])(f(x)>q)$ is equivalent to $(\exists r\in [a,b]\cap \Q)(f(r)>q)$ for quasi-continuous $f:\R\di \R$.  
Hence, the supremum $\sup_{x\in [a,b]}f(x)$ can be defined using $(\exists^{2})$ and a dominating function for $f$ as in the third item is then straightforward. 

\smallskip

For the final sentence, work in $\RCAo+\QFAC^{0,1}$ and invoke classical logic in the form of $(\exists^{2})\vee \neg(\exists^{2})$.  If $(\exists^{2})$, proceed as in the previous paragraph.
If $\neg(\exists^{2})$, all functions on $\R$ are continuous by \cite{kohlenbach2}*{Prop.\ 3.14}, i.e.\ there is nothing to prove.  
\end{proof}
We now establish the following equivalences.  On one hand, the items in Theorem~\ref{Taranie} are quite strong as they imply $\ATR_{0}$ by Theorem~\ref{hofff}. 
On the other hand, Theorem \ref{weaklame} implies that the restriction of \eqref{slofi} and \eqref{slofidva} to the unit interval is weak.  
\begin{thm}[$\ACAo+\QFAC^{0,1}$] \label{Taranie}
The following are equivalent.
\begin{enumerate} 
\renewcommand{\theenumi}{\alph{enumi}}
\item The bounded numerical choice principle $\BNC^{\omega}$.  \label{fdZ}
\item For sub-continuous $f:\R\di \R$, there is continuous $g:\R\di \R$ with $(\forall x\in \R)(f(x)\leq g(x))$.\label{slofi}
\item For everywhere Riemann integrable $f:\R\di \R$, there is continuous $g:\R\di \R$ with $(\forall x\in \R)(f(x)\leq g(x))$.\label{slofidva}
\item For everywhere Riemann integrable $f:\R\di \R$, there is a modulus of Riemann integration.\label{sokolz}
\item For a continuous ae function $f:\R\di \R$ that is bounded on every closed interval, there is a modulus of Riemann integration.\label{sokolz2}
\item Item \eqref{sokolz} restricted to $f:[0,1]\di \R$. \label{sokolz3}
\item Item \eqref{sokolz2} restricted to $f:[0,1]\di \R$. \label{sokolz4}
\item Item \eqref{slofi} for `sub-continuous' replaced by `locally bounded'.\label{Zachionz}
\item Any of items \eqref{slofi}-\eqref{Zachionz} restricted to simply continuous functions.  \label{zachionz}
\item Any of items \eqref{slofi}-\eqref{Zachionz} restricted to \(Baire\) measurable functions.  \label{zachionz2}
\end{enumerate}
\end{thm}
\begin{proof}
To prove items \eqref{slofi}-\eqref{zachionz2}, for sub-continuous and Riemann integrable functions $f:\R\di \R$, Theorem \ref{weaklame} implies the following:
\be\label{dino}
(\forall n\in \N)(\exists m\in \N)(\forall x\in [-n,n])(f(x)\leq m).
\ee
Apply item \eqref{fdZ} to \eqref{dino} and use $(\exists^{2})$ to modify the resulting sequence into a continuous function.  
For item \eqref{sokolz}, apply the former principle to the epsilon-delta definition of `Riemann integrability everywhere'.  
Item \eqref{sokolz2} then follows by Theorem \ref{bengas2}.  

\smallskip

To derive $\BNC^{\omega}$ from item \eqref{Zachionz}, let $\psi$ be arithmetical and such that $(\forall n\in \N)(\exists m\in \N)(\forall f\in 2^{\N})\psi(n, m, f)$.  We may further assume 
w.l.o.g.\ that 
\be\label{combine_f}
(\forall f_1 ,  f_2 \in 2^{\N})(\exists f_{3} \in 2^{\N}) (\forall  m \in \N) \big[ \psi(n, m, f_3) \asa  [\psi(n, m, f_1) \land \psi(n, m, f_2)]\big].
\ee
with an effective way to find an $f_3$ given $f_1$ and $f_2$.  Now define $f:\R\di \R$ as 
\be\label{counter}
f(x)= (\mu m) \psi(\lfloor |x|\rfloor,m,\xi(x))
\ee
where $\xi(x)$ is the binary representation of $|x|-\lfloor |x|\rfloor$, choosing a tail of zeros if relevant, or use the Cantor set construction below.  
By assumption, $f$ is locally bounded as it has an upper bound on each $[n, n+1]$.  Let $g:\R\di \R$ be as provided by item \eqref{Zachionz}.  
Since the former is continuous, $(\exists^{2})$ yields $h(n):= \lceil\sup_{x\in [-n,n]}g(x)\rceil $ by \cite{kohlenbach2}*{\S3}.  
Essentially by definition, using induction or choice to get a high enough value through \eqref{combine_f}, $h$ is also the choice function required by item \eqref{fdZ}.  
Sub-continuity trivially follows from local boundedness, i.e.\ item \eqref{slofi} also suffices.  

\smallskip

To show that item \eqref{zachionz} also implies item~\eqref{fdZ}, let $\mathcal{C}\subset [0,1]$ be the well-known Cantor set and define $\tilde{f}:\R\di \R$ as in \eqref{counter} but restricted to $(|x|-\lfloor |x|\rfloor)\in \mathcal{C}$ (and zero otherwise).  We now show that $\tilde{f}$ is both sub- and simply continuous. 
Since $\mathcal{C}$ is closed, $\tilde{f}$ is zero close enough to any $x$ with $(|x|-\lfloor |x|\rfloor)\in \mathcal{C}^{\textsf{c}}$ and hence continuous at $x$. 
In case $x$ is such that $(|x|-\lfloor |x|\rfloor)\in \mathcal{C}$, let $(x_{n})_{n\in \N}$ be any sequence converging to $x$.  By definition, there is $m_{0}\in \N$ such that for all $n\in \N$, we have $\tilde{f}(x_{n})\in \{ 0, \dots, m_{0}\}$.  
Hence, the sequence $(f(x_{n}))_{n\in \N}$ has a constant (and hence convergent) sub-sequence, i.e.\ $\tilde{f}$ is sub-continuous.  
For simple continuity, observe that $\tilde{f}^{-1}(V)=\mathcal{C}^{\textsf{c}}\cup U$ for open $V\subset \R$ with $0\in V$ and $U\subset \mathcal{C}$ is nowhere dense.   
In case $0\not \in V$, then $\tilde{f}^{-1}(V)= {\emptyset} \cup U=U\subset \mathcal{C}$ is also nowhere dense and as required for simple continuity (and Baire measurability).  
Since $\mathcal{C}$ is closed and measure zero, $\tilde{f}$ is measurable, i.e.\ item \eqref{zachionz2} implies item~\eqref{fdZ}.  
Using the epsilon-delta definition, $\tilde{f}$ is also Riemann integrable everywhere (and continuous ae and bounded on every interval) with zero integral.  We use induction or choice to convert boundedness on unit intervals to boundedness on all intervals.  
Thus, item \eqref{slofidva} also implies item \eqref{fdZ} and the same for the restrictions of the former.  
For items \eqref{sokolz} and \eqref{sokolz2}, Footnote \ref{uftring} shows that a modulus of (everywhere) Riemann integration for $g:\R\di \R$ yields a sequence $(N_{n})_{n\in \N}$ of upper bounds for $g$ on $[-n,n]$.  
Applied to $\tilde{f}$, one immediately obtains item \eqref{fdZ}.

\smallskip

For items $\eqref{sokolz3}$ and $\eqref{sokolz4}$, we follow the above argument using $\mathcal{C}$ but `squash' the function $f$ as follows.  Instead of defining $f(x)=m$ (or $m+1$), we select about $m$-many points (for each $x$ and $n$) and set $f$ to $1/(n+1)$ there.  For example, let $(q_n)_{n\in \N}$ be a fixed enumeration of $\Q$ and define $f((x+q_i) \, \mod \, 1) := 1/(n+1)$, where $i<m$ and conflicts are resolved by preferring higher values of $f$.  
Note that we use induction or choice to get boundedness of $m$ if $n$ is bounded.  Then the new function is continuous except (possibly) at rational translations ($\mod \, 1$) of $\mathcal{C}$.
\end{proof}
Next, we discuss some remarks on generalisations and variations of Theorem \ref{Taranie}.
Remark \ref{doooonggg} is motivated by Montalb\'an's notion of robustness in RM (\cite{montahue}).  
\begin{rem}\label{bipoli}\rm
First of all, for $\eqref{slofi} \lor \eqref{Zachionz} \di \eqref{fdZ}$ in Theorem \ref{Taranie} \(including with $\eqref{zachionz}$ and $\eqref{zachionz2}$\), $\QFAC^{0,1}$ can be replaced by $\SIND$.  
For the other items, to obtain item~$\eqref{fdZ}$, $\QFAC^{0,1}$ can be replaced by induction.  In particular, we only need the induction axiom for $(\forall m \in \N) (\exists f \in \N^{\N})  \varphi(f, m, n)$ with $\varphi$ arithmetical, as is readily verified.  

\smallskip 

Secondly, in items \eqref{sokolz} and $\eqref{sokolz2}$ of Theorem \ref{Taranie}, the existence of a modulus of Riemann integration \emph{cannot} be weakened to the existence of Riemann integral as a function.  The minimal model of $\ACAo$ containing all hyperarithmetic reals is a counterexample.  Uniformly, the measure of a $\Sigma^0_n$-set is $\Sigma^0_n$, which allows for obtaining the Riemann integral without obtaining a modulus of Riemann integration.

\smallskip

Thirdly, there is nothing special about $\R$: Theorem \ref{Taranie} generalises to Euclidean space via minimal adaptation.  The generalisation to metric spaces is non-trivial, as discussed in Section \ref{fetric}.    
In particular, the latter shows that various basic properties of metric spaces, e.g.\ pertaining to continuous functions, are equivalent to $\BNC^{\omega}$. 

\smallskip

Fourth, the existence of a modulus of Riemann integrability is non-trivial by Theorem \ref{Taranie}.  However, for a function of bounded variation $f:[0,1]\di \R$, such a modulus is readily defined.  
Indeed, given total variation $\delta$ and mesh $\eps$, the difference between the associated Riemann sums is bounded by $\delta\eps$.  
Hence, we cannot restrict Theorem \ref{Taranie} to the former function class, in contrast to the results in \cite{dagsamXI, samBIG, samBIG2} where bounded variation plays a central role in the study of principles similar in strength to $\BNC^{\omega}$, namely at the level of $\ATR_{0}$. 
\end{rem}
\begin{rem}[Robustness]\label{doooonggg}\rm
First of all, item \eqref{sokolz2} in Theorem \ref{Taranie} deals with `continuity ae', but slight restrictions can be imposed:  one can replace `measure zero' by `effectively measure zero' (see \cite{avi1337, nieyo}), where the latter means that 
there is a sequence $(I_{n,k})_{n,k\in \N}$ of open intervals such that for fixed $k\in \N$, the union $\cup_{n\in \N}I_{n, k}$ has total length at most $\frac{1}{2^{k}}$ and covers the set at hand.  
Clearly, the function $\tilde{f}$ from the proof of Theorem \ref{Taranie} is continuous outside an effectively measure zero set, as the discontinuity set of $\tilde{f}$ is included in the countable union of translations of $\mathcal{C}$.  

\smallskip

Secondly, Theorem \ref{Taranie} goes through for sub-continuity replaced by Cauchy-sub-regularity, the open-cover formulation of sub-continuity, local compactness, CC-regularity, and uniform local boundedness (see \cites{guptak, systemenouveau} for definitions).  These are admittedly merely variations on a theme, but do establish robustness.    

\smallskip

Thirdly, item \eqref{fdZ} of Theorem \ref{Taranie} readily follows from the following formulation of the Darboux criterion for Riemann integrability (see e.g.\ \cite{taoana1}*{\S11.3}), also similar to the \emph{squeeze theorem} (see e.g.\ \cite{bartle2}).
\begin{center}
\emph{For everywhere Riemann integrable $f:\R \di \R$ and $\eps>0$, there are step functions $s, t:\R\di \R$ such that $s(x)\leq f(x)\leq t(x)$ for all $x\in \R$ and $\int_{a}^{b}(t(x)-s(x))dx <\eps$ for all $a, b\in \R$.}
\end{center}
Hence, the previous squeeze theorem also implies the strong system $\ATR_{0}$.  

\smallskip

Fourth, Riemann integrable functions are well-known to be Lebesgue integrable.  
One definition of the latter integral (for non-negative functions) proceeds via lower and upper integrals, which in turn are the infimum and supremum of the integrals of simple functions that respectively dominate and majorise the function at hand.   Similarly, $L^{1}$-functions on Euclidean space can be approximated via step functions (\cite{taomes}*{Theorem 1.3.20}).  
These observations should yield plenty of theorems that imply $\BNC^{\omega}$, if we require some additional regularity properties.  

\smallskip

Finally, the above squeeze theorem resembles the Vitali-Carath\'eodory theorem (\cite{rudin}*{Theorem 2.25}).  However, the latter involves a lower semi-continuous dominating function.  
The principle $\ENC^{\omega}$ implies Tao's pigeon principle for the Lebesgue measure (\cite{taomes}), called $\PHP_{[0,1]}$ in \cite{samBIG2}.  
Moreover, assuming the latter and $\BNC^{\omega}$, one proves the Vitali-Carath\'eodory theorem for Riemann integrable functions.  
\end{rem}
In conclusion, there are a number of natural equivalences for $\BNC^{\omega}$, with many promising and robust variations. 

\subsection{Numerical Choice and metric spaces}\label{fetric}
We establish equivalences between $\BNC^{\omega}$ and basic properties of metric spaces.  We briefly introduce the latter in Section \ref{defmet} and obtain our results in Section \ref{maint}.
A more detailed introduction to metric spaces in higher-order RM may be found in \cites{samBOOK, sammetric, samHARD}. 
\subsubsection{Metric spaces and higher-order arithmetic}\label{defmet}
We introduce the well-known definition of metric space $(M, d)$ to be used in the below, namely Definition \ref{donkc}.  
The latter is really the textbook definition with some details like function extensionality made explicit. 
We shall generally only study the case where $M$ is a subset of $\R$, up to coding of finite sequences.  

\smallskip

Firstly, in our study of metric spaces $(M, d)$, we assume that the set $M$ comes with its own equivalence relation `$=_{M}$' and that the metric $d:M^{2}\di \R$ satisfies 
the axiom of extensionality relative to this relation as follows:
\[
(\forall x, y, v, w\in M)\big([x=_{M}y\wedge v=_{M}w]\di d(x, v)=_{\R}d(y, w)\big).
\]
Similarly to functions on the reals, `$F:M\di \R$' denotes a function that satisfies the following instance of the axiom of function extensionality:
\be\tag{\textup{\textsf{E}}$_{M}$}\label{koooooo}
(\forall x, y\in M)(x=_{M}y\di F(x)=_{\R}F(y)).
\ee
We recall that the study of metric space in second-order RM is at its core based on equivalence relations, as discussed explicitly in e.g.\ \cite{simpson2}*{I.4.3} or \cite{damurm}*{\S10.1}.   

\smallskip

Secondly, we now have the following (textbook) definition of metric space. 
\bdefi[$\RCAo$]\label{donkc}
A functional $d: M^{2}\di \R$ is a \emph{metric on $M$} if it satisfies the following properties for $x, y, z\in M$:
\begin{enumerate}
 \renewcommand{\theenumi}{\alph{enumi}}
\item $d(x, y)=_{\R}0 \asa  x=_{M}y$,
\item $0\leq_{\R} d(x, y)=_{\R}d(y, x), $
\item $d(x, y)\leq_{\R} d(x, z)+ d(z, y)$.
\end{enumerate}
\edefi
To be absolutely clear, we shall study metric spaces $(M, d)$ with $M\subset \N^{\N}$ or $M\subset \R$, unless explicitly stated otherwise. 
Thus, quantifying over $M$ amounts to quantifying over $\N^{\N}$ or $\R$, perhaps modulo coding of finite sequences, i.e.\ the previous definition can be made in third-order arithmetic.   For \emph{compact} metric spaces, this restriction is minimal as the cardinality of such spaces is at most that of the continuum by \cite{buko}*{Theorem 3.13}.  We mostly study countable unions of compact spaces in this paper.  

\smallskip

Next, Definition~\ref{char} generalises as follows, keeping in mind \eqref{koooooo}.  
\bdefi[Sets]\label{charc} Let $(M, d)$ be a metric space.  
\begin{enumerate}
\renewcommand{\theenumi}{\alph{enumi}}
\item A set $A\subset M$ is given by a function $F_{A}:M\di \{0,1\}$; we write $x\in A$ for $ F_{A}(x)=1$, for any $x\in \R$.\label{tinkzzzx}
\item The set $B_{d}^{M}(x, r)$ denotes the open ball $\{y\in M: d(x, y)<_{\R}r\}$.  We write $|B_{d}^{M}(x, r)|=2r$ to denote its diameter.   
\item A set $A\subset M$ is \emph{finite} if there is $N\in \N$ such that for any finite sequence $x_{0}, \dots, x_{N}$ of distinct\footnote{An important technicality is that `distinct' is interpreted relative to `$=_{M}$'.  Indeed, since there are uncountably many Cauchy sequences that converge (fast) to $0$, the set $\{0\}$ is only finite if `distinct' is defined via `$=_{\R}$'.  Moreover, if $\SIND$ is available, one can simplify the definition of `$A$ is finite' to the statement that for some $N\in \N$, $A$ does not contain $N$ distinct elements.  The above definition (on the reals) is needed in the RM of e.g.\ Fourier analysis (\cite{samBOOK, samBIG}). \label{zerefl}} elements of $M$, there is $i\leq N$ such that $x_{i}\not\in A$.  The number $N$ is a `size bound of $A$', also denoted `$|A|\leq N$'.  
\item A set $A\subset M$ is \emph{cuf} (countable union of finite sets, \cite{heerlijkheid}) if there is a sequence of finite sets $(A_{n})_{n\in \N}$ with $A=\cup_{n\in \N}A_{n}$.  
\end{enumerate}
\edefi
Other notions (open set, continuity, \dots) now generalise in the usual way.  Basic properties of cuf sets, like enumerability or measure zero, imply Feferman's projection principle and $\FIVE$ (\cite{samHARD, samBOOK, samCREP}).  
By Theorem \ref{tomore} below, $\BNC^{\omega}$ is equivalent to a cuf set $\cup_{n\in \N}A_{n}$ having a size bound, i.e.\ there is $g\in \N^{\N}$ with $|A_{n}|\leq g(n)$.  
The latter is a rather weak property, yet yields $\ATR_{0}$.    

\smallskip

Thirdly, the following definitions are now standard, where we note that a different nomenclature is sometimes used in second-order RM.  
\bdefi\label{deacop} 
For a metric space $(M, d)$, we say that
\begin{itemize}
\item $(M, d)$ is \emph{countably-compact} if for any sequence $(O_{n})_{n\in \N}$ of open sets in $M$ such that $M\subset \cup_{n\in \N}O_{n}$, there is $m\in \N$ such that  $M\subset \cup_{n\leq m}O_{n}$,
\item $(M, d)$ is \emph{compact} in case for any $\Psi:M\di \R^{+}$, there are $x_{0}, \dots, x_{k}\in M$ such that $\cup_{i\leq k}B_{d}^{M}(x_{i}, \Psi(x_{i}))$ covers $M$, 
\item $(M, d)$ is \emph{sequentially compact} if any sequence has a convergent sub-sequence,
\item $(M, d)$ is \emph{bounded} if there is $N\in \N$ such that $d(x, y)<N$ for all $x, y\in M$.
\end{itemize}
\edefi
There are a number of compactness notions for metric spaces, many of which yield a generalisation of Theorem \ref{tomore}.  

\subsubsection{Equivalences involving metric spaces}\label{maint}
We establish some equivalences for $\BNC^{\omega}$ and basic properties of metric spaces $(M, d)$, especially $\sigma$-compactness.  

\smallskip

First of all, $\QFAC^{0,1}$ boasts many equivalences involving compact metric spaces (\cite{samBOOK, samCREP}), including Dini's theorem and the boundedness theorem for continuous functions.  
By the following theorem, sequential versions of such principles are equivalent to $\BNC^{\omega}$, akin to the case for the reals.    
Item~\eqref{tamy1a} should be contrasted with \cite{simpson2}*{IV.1.7} whereby a version of the former for codes is provable in $\WKL_{0}$.  
\begin{thm}[$\ACAo+\QFAC^{0,1}+\SIND$]\label{tomore}
The following are equivalent.
\begin{enumerate}
\renewcommand{\theenumi}{\alph{enumi}}
\item Numerical choice as in $\BNC^{\omega}$.
\item Let $(M_{n}, d)$ be sequentially compact for all $n\in \N$.  There is $g\in \N^{\N}$ such that $(M_{n}, d)$ is bounded by $g(n)$ for each $n\in \N$.  \label{tamy1}
\item The previous item with `sequentially' removed.\label{tamy2}
\item \(Heine-Borel\) Let $(M_{n}, d)$ be sequentially compact for all $n\in \N$ and let $(O_{m})_{m\in \N}$ be an open cover of $\cup_{n\in \N}M_{n}$.  There is $g\in \N^{\N}$ such that $(M_{n}, d)$ is covered by $\cup_{m\leq g(n)}O_{n}$ for each $n\in \N$.  \label{tamy1a}
\item Let $(M, d)$ be sequentially compact and let $(f_{n})_{n\in \N}$ be a sequence of continuous $M\di \R$-functions.  Then there is $g\in \N^{\N}$ such that $(\forall x\in M, n\in \N)(f_{n}(x)\leq g(n))$.\label{tamy3}
\item The previous item restricted to Lipschitz continuous functions.\label{tamy4}
\item Item \eqref{tamy3} generalised to sub-continuous functions.\label{tamy5}
\item \(Dini\) Let $(M_{n}, d)$ be sequentially compact for all $n\in \N$.  Let $(f_{m})_{m\in \N}$ be a monotone sequence of continuous functions on $M=\cup_{n}M_{n}$ converging to continuous $f:M\di \R$.  
Then the convergence is semi-uniform, i.e.\ there is $g:\N^{2}\di \N$ such that\label{zefram}
\be\textstyle\label{juxta} 
(\forall k, n\in \N)(\forall m\geq g(n,k))(\forall x\in M_{n})(|f_{m}(x)-f(x)|<\frac{1}{2^{k}}). 
\ee
\item For sequentially compact $(M, d)$ and cuf $A=\cup_{n}A_{n}$ with $A_{n}\subset M$ finite, there is a a uniform size bound, i.e.\ $g\in \N^{\N}$ with $|A_{n}|\leq g(n)$ for all $n\in \N$.\label{glimmer2}
\item For sequentially compact $(M, d)$ and cuf $A=\cup_{n}A_{n}$ with $A_{n}\subset M$ finite, there is $g\in \N^{\N}$ with $\| A_{n}\|:=\max_{x, y\in A_{n}}d(x, y)$ at most $g(n)$ for all $n\in \N$.\label{glimmer}
\end{enumerate}
\end{thm}
\begin{proof}
First of all, assume $\BNC^{\omega}$ and let $(M_{n}, d)$ be compact for all $n\in \N$.  Compactness trivially implies boundedness and apply the former to $(\forall n\in \N)(\exists N\in \N)(\forall f, g\in M_{n})( d(f, g)\leq  N )  $ to obtain item \eqref{tamy2}.  Using $\QFAC^{0,1}$, sequential compactness also implies boundedness and countable compactness (via the obvious proof-by-contradiction), i.e.\ items~\eqref{tamy1} and \eqref{tamy1a} also follows.  
For items \eqref{tamy3} and \eqref{tamy4}, the usual proof-by-contradiction (again using $\QFAC^{0,1}$) shows that a (sub-)continuous function on a sequentially compact metric space is bounded.  
Apply $\BNC^{\omega}$ to $(\forall n\in \N)(\exists N\in \N)(\forall f\in M_{n})(|F(f)|\leq N )$ where $F:M\di \R$ is continuous.
To establish item \eqref{zefram}, the usual proof-by-contradiction) of Dini's theorem works in $\ACAo+\QFAC^{0,1}$; then apply $\BNC^{\omega}$ to obtain \eqref{juxta}.   
Items \eqref{glimmer2} and \eqref{glimmer} have an obvious proof given their syntactic form.  

\smallskip

Secondly, to derive $\BNC^{\omega}$ from the other items, let $\varphi$ be arithmetical and such that $(\forall n\in \N)(\exists m\in \N)(\forall f\in \N^{\N})\varphi(f,m,n)$. 
Define $Z:M\di \R$ as follows
\[
Z(f):= 
\begin{cases}
m+1 &\begin{array}{l} \textup{if $m\in \N$ is the least number such that} \\ \textup{$\varphi(f(1)*f(2)*\dots, m, f(0))$ and $f\ne_{1} 00\dots$}\end{array}\\
0 &  \textup{ if $f=_{1}00\dots$}
\end{cases}.
\]
Now observe that $d(f, g)=|Z(f)-Z(g)|$ is a metric on $M=\N^{\N}$ if we define $f=_{M}g$ by $Z(f)=_{0}Z(g)$.  
Also observe that $M_{n}=\{f\in M:f(0)\leq n\}$ is (sequentially) compact.  
Let $g\in \N^{\N}$ be as provided by item~\eqref{tamy1} or \eqref{tamy2}, implying $(\forall f\in M_{n})(Z(f)=d(f, 00\dots)<g(n))$. 
Hence, we obtain a function as in the consequent of $\BNC^{\omega}$.  

\smallskip

For item \eqref{tamy1a}, define the open set $O_{m}=\{f\in M: Z(f)=m\}$ and observe that $\BNC^{\omega}$ follows from the former. 
To obtain $\BNC^{\omega}$ from item \eqref{tamy3} or \eqref{tamy4}, note that $Z:M\di \R$ is Lipschitz continuous (as an $M\di \R$-function), i.e.\ a bound for this function also yields $\BNC^{\omega}$. 
Now assume item \eqref{zefram} and define $Z_{m}:M\di \R$ as $Z(f)$ if $Z(f)\leq m$, and $0$ otherwise.  
Then $(Z_{m})_{m\in \N}$ is a monotone sequence of continuous functions with limit $Z$.  
Let $g:\N^{2}\di \N$ be as in \eqref{juxta} and note that we readily obtain the function required by $\BNC^{\omega}$. 

\smallskip

To show that item \eqref{glimmer} implies $\BNC^{\omega}$ observe that $M=\cup_{n\in \N}M_{n}$ is cuf as each $M_{n}$ contains at most finitely many elements, where we recall Footnote \ref{zerefl}.
Clearly, any function $g\in \N^{\N}$ with $\|M_{n}\|\leq g(n)$ yields the choice function from the consequent of $\BNC^{\omega}$.
Unfortunately, an upper bound on the size $|M_{n}|$ does not yield the same information.  
To remedy this, consider the following functional similar to $Z$:
\[
W(f):= 
\begin{cases}
i+1 &\begin{array}{l} \textup{if $f\ne_{1} 00\dots$ and $m\in \N$ is the least number such that} \\ \textup{$\varphi(f(2)*f(3)*\dots, m, f(0))$ and $f(1)=i\leq m$}\end{array}\\
0 &  \textup{ if $f=_{1}00\dots$}
\end{cases}.
\]
Then $W$ yields a metric space $(M, d')$ in the same way as for $Z$.  The set $M$ is cuf as in $M=\cup_{n\in \N}N_{n}$ where $N_{n}:=\{f\in M: f(0)\leq n \}$.    
If $g\in \N^{\N}$ is such that $|N_{n}|\leq g(n)$ for all $n\in \N$, we do obtain the consequent of $\BNC^{\omega}$.
\end{proof}
As a side result, we observe that e.g.\ $\BNC^{\omega}+\QFAC^{0,1}$ is equivalent to item \eqref{tamy1} over $\ACAo+\SIND$.
We also discuss a further variation of the above. 
\begin{rem}\rm
Let $(X, \preceq_{X})$ be a linear ordering such that $X\subset \N^{\N}$ is cuf, i.e.\ $X=\cup_{n\in \N}X_{n}$ where each $X_{n}$ is finite; the latter is as in Definition~\ref{charc}: there is some $N\in \N$ such that for any $N+1$ `distinct relative to $=_{X}$' elements, there is at least one not in $X_{n}$.  Then $\BNC^{\omega}$ is equivalent to the existence of $g\in \N^{\N}$ with $|X_{n}|\leq g(n)$ for all $n\in \N$ and any such ordering.    
%
%
%
\end{rem}
In conclusion, a number of basic statements about metric spaces are equivalent to numerical choice as in $\BNC^{\omega}$, perhaps most noteworthy being the principle stating that cuf sets have a uniform size bound; the latter property seems much weaker than an enumeration. 

\subsection{There and back again}
We show that equivalences for higher-order  numerical choice $\BNC^{\omega}$ can be `pushed down' to yield equivalences involving second-order numerical choice $\BNC$.  
Hence, higher-order RM directly contributes to second-order RM, which is a theme explored in \cite{samBOOK}.  On a technical note, we shall use `arithmetical' to refer to formulas only involving arithmetical quantifiers, while `$\L_{2}$-arithmetical' means that the parameters are further restricted to second-order.

\smallskip

First of all, we push down Theorem \ref{Taranie} to obtain equivalences for $\ATR_{0}$.
A function is \emph{effectively Baire 2} if it is the iterated limit of a double sequence of continuous functions.  
\begin{thm}[$\ACAo+\SIND$]\label{coredesigner}
The following are equivalent.
\begin{enumerate}
\renewcommand{\theenumi}{\alph{enumi}}
\item For sub-continuous and effectively Baire 2 $f:\R\di \R$, there is continuous $g:\R\di \R$ with $(\forall x\in \R)(f(x)\leq g(x))$.\label{slofi1337}
\item For Riemann integrable and effectively Baire 2 $f:\R\di \R$, there is continuous $g:\R\di \R$ with $(\forall x\in \R)(f(x)\leq g(x))$.\label{slofi13372}
\item Item \eqref{slofi1337} restricted to simply continuous functions.  \label{zachionz1337}
\item Item \eqref{slofi1337} restricted to measurable functions.  \label{zachionz21337}
\item For a sub-continuous $\bf\Delta^1_1$-function $f:\R\di \R$, there is continuous $g:\R\di \R$ with $(\forall x\in \R)(f(x)\leq g(x))$. \label{item_Delta11}
\item \($\BNC$\)\label{xionz}
For any $\varphi \in \Pi_{1}^{1}$, we have that
\[
(\forall n\in \N)(\exists m\in \N)\varphi(n, m)\di  (\exists g\in \N^{\N})(\forall n\in \N)(\exists m\leq g(n))\varphi(n, m).
\]
\item The system $\ATR_{0}$.
\end{enumerate}
\end{thm}
\begin{proof}
The final two items are equivalent over $\RCA_{0}$ by Theorem \ref{hofff}.  To obtain items \eqref{slofi1337}-\eqref{item_Delta11} (even without $\SIND$), using $\SAC$  (implied by $\ATR_{0}$), $f$ is bounded on every interval, and we can apply $\BNC$.

\smallskip

That each of \eqref{slofi1337}-\eqref{item_Delta11} implies $\ATR_0$ is an easy consequence of item \eqref{kuhl} in Theorem~\ref{hofff}.  Here, $\SIND$ is only used for $\eqref{slofi13372}$ to get boundedness on every finite (as opposed to unit) interval.  Moreover, $\SIND$ suffices since we only need at most one non-zero value per unit interval, so we can enumerate these values for an interval.

\smallskip

We also provide a different proof that does not rely on Theorem $\ref{hofff} \eqref{kuhl}$ and thus might be more generalisable to third-order analogues.
To show that item \eqref{slofi1337} implies item \eqref{xionz}, we modify the proof of Theorem~\ref{Taranie}.  
In particular, we show that $f$ from \eqref{counter} and $\tilde{f}$ are effectively Baire 2 in case $\psi$ is arithmetical and only has second-order parameters.  
To this end, we assume
\be\label{gezelles}
(\forall a \in \N)(\exists b\in \N)(\forall x \in 2^\N) \big[(\forall m\in \N) (\exists n\in \N) R(a,b,x,m,n)\big],
\ee
where $R$ is primitive recursive, i.e.\ \eqref{gezelles} is part of $\L_{2}$.  We now construct continuous functions $H_{n,m}:2^\N \rightarrow \R$ such that the double limit $H = \lim_{m \rinf}\lim_{n \rinf}H_{n,m}$ is well-defined.  
Identifying $2^\N$ with the Cantor set, we extend each $H_{n,m}$ to a continuous function $h_{n,m}:[0,1] \rightarrow \R$ by extending the graph with straight lines. 
Note that all limits commute with this extension and that the corresponding extension $h$ of $H$ is effectively Baire 2.
We first introduce the following function
\be\label{fatherly}
G_{a,n,m}(x) := 
\begin{cases}
b+1 & \begin{array}{l} \textup{if $b\leq m$ is the least number with} \\ \textup{$(\forall i \leq m)( \exists j \leq n) R(a,b,x,i,j)$}\end{array}\\
0 & \textup{otherwise}. 
\end{cases}.
\ee
We now define $H_{n,m}:2^{\N}\di \R$ by cases as follows.
\be\label{zonkojoke}
\begin{array}{l}
\text{If the real $x$ is of the form $\underbrace{1*\dots *1}_{\textup{$m+1$ times}}*~y$, we define $H_{n,m}(x): = 0$.}\\
\text{If for some $a \leq m$, the real $x$ is of the form $\underbrace{1*\dots *1}_{\textup{$a$ times}}*~0* y$,}\\
\text{we define $H_{n,m}(x) :=  G_{a,n,m}(y)$}.
\end{array}
\ee
By Kleene's normal form theorem (provable in $\ACA_{0}$; \cite{simpson2}), the variable $x$ in \eqref{fatherly} may be assumed to occur as $\overline{x}n$, which guarantees that $H_{n,m}$ is continuous.  
Letting $H$ be the double limit of the former, we have $H(x) = b+1$ if $b\in \N$ is the least number with $(\forall m\in \N) (\exists n\in \N) R(a,b,y,m,n)$, and 0 otherwise.
Let $h$ be the extension of $H$ as above and observe that it is as required.  
\end{proof}
Secondly, the previous equivalences are mere examples, i.e.\ there should be many variations and equivalences. 
To appreciate the below generalisations, we consider the following version of topological continuity.
\bdefi\label{fefetieve}
A function $f:\R\di \R$ is \emph{effectively continuous} if {for any sequence $(V_{n})_{n\in \N}\subset \R$ of RM-open sets, $(f^{-1}(V_{n}))_{n\in \N}$ is a sequence of RM-open sets.}
\edefi
The adjective `effectively' in the previous definition is apt by the following theorem.  Indeed, a continuous modulus suffices to obtain an RM-code (\cite{dagsamXIV}*{\S2}).   
\begin{thm}[$\RCAo$]\label{dirfgje}
An effectively continuous function $f:\R\di \R$ has a continuous modulus of continuity.  
\end{thm}
\begin{proof}
Let $f:\R\di\R$ be effectively continuous and define a sequence of RM-open sets as $V_{n, k}=B(f(q_{n}), \frac{1}{2^{k}})$ where $(q_{n})_{n\in \N}$ is a fixed enumeration of the rationals. 
Let $(O_{n,k})_{n, k \in \N}$ be a sequence of RM-open sets such that $O_{n, k}=f^{-1}(V_{n, k})$ for all $n, k\in \N$.  
In particular, $O_{n, k}=\cup_{m\in \N}(a_{n,k,m}, b_{n,k,m})$ for triple sequences of reals as stated.  
By continuity, for any $x\in \R$ and $k\in \N$, there is $n\in \N$ with $|f(x)-f(q_{n})|<\frac{1}{2^{k}}$.  
The latter formula is `$x\in f^{-1}(V_{n,k})$', which is equivalent to $(\exists m\in \N)(x\in (a_{n,k, m} , b_{n,k,m} ) )$, yielding 
\be\label{koupap}\textstyle 
(\forall x\in \R, k\in \N)(\exists N, m, n\in \N)(a_{n,k, m} <x-\frac{1}{2^{N}}< x+\frac{1}{2^{N}}< b_{n,k,m} ) ).
\ee
Let $\Phi$ be the functional resulting from applying $\QFAC^{1, 0}$ to \eqref{koupap}.  
In case $\Phi$ is continuous, it yields the required modulus.  In case $\Phi$ is discontinuous, we obtain $(\exists^{2})$ by \cite{kohlenbach2}*{Prop.\ 3.14}.
The latter readily yields a modulus of continuity (\cite{kohlenbach4}*{\S4}) by restricting the innermost quantifier in the definition of continuity to the rationals. 
An unbounded search then yields the required modulus.  
\end{proof}
Thirdly, inspired by Definition \ref{fefetieve}, we define additional `effective' definitions. 
\begin{defi}[Effective definitions]\label{flung2} For a function $f:\R\di \R$, we have the following definitions:
\begin{itemize}
\item $f$ is \textup{EB2} if for all sequences $(G_{n})_{n\in \N}$ of RM-opens, there is a sequence $(O_{n,m})_{n\in \N}$ of RM-open sets with $f^{-1}(G_{n})=\cup_{n\in \N}\cap_{m\in \N}O_{n,m}$ for $n\in \N$,
\item $f$ is \emph{arithmetically effectively measurable} if for any sequence $(G_{n})_{n\in \N}$ of RM-open sets, there is a sequence $(C_{n,m}, N_{n})_{n,m\in \N}$ such that for all $n\in \N$, 
we have that $f^{-1}(G_{n}) $ equals $ \cup_{m\in \N}C_{n,m} \bigcup N_{n}$, that $C_{n,m}$ is RM-closed, and that $N_{n}$ is $\L_{2}$-arithmetical and measure zero,
\item $f$ is \emph{arithmetically effectively sico} if for any sequence $(G_{n})_{n\in \N}$ of RM-open sets, there is a sequence $(O_{n}, N_{n})_{n\in \N}$ such that for all $n\in \N$, $f^{-1}(G_{n})=O_{n}\cup N_{n}$, 
$O_{n}$ is RM-open, and $N_{n}$ is $\L_{2}$-arithmetical and nowhere dense,
\item $f$ is \emph{arithmetically effectively Baire measurable} if for any sequence $(G_{n})_{n\in \N}$ of RM-open sets, there is a sequence $(O_{n}, N_{n,m})_{n,m\in \N}$ such that for all $n\in \N$, we have that $f^{-1}(G_{n})\Delta O_{n}= N_{n} $, each $O_{n}$ is RM-open, and $N_{n}=\cup_{m\in \N}N_{n,m}$ where each $N_{n,m}$ is nowhere dense and $\L_{2}$-arithmetical.  
\end{itemize}
\end{defi}
While `Baire 2', `effectively Baire 2' and `EB2' are all equivalent over $\ZFC$, this cannot be proved in $\Z_{2}^{\omega}+\QFAC^{0,1}$ (\cite{dagsamXIV}).

\smallskip

We now have the following corollary to Theorem \ref{coredesigner}.
\begin{cor}\label{zinga}
Theorem \ref{coredesigner} goes through for `effectively Baire 2' replaced by any of the items from Definition \ref{flung2}.
\end{cor}
\begin{proof}
Observe that the proof of Theorem \ref{coredesigner} goes through.  In particular, the function $h$ satisfies all effective notions from Defintion \ref{flung2}.  
\end{proof}
In conclusion, the various effective definitions from Definition \ref{flung2} push down higher-order equivalences for $\BNC^{\omega}$ to equivalences for $\ATR_{0}$. 

\section{Connections}\label{duko}
\subsection{Introduction}
In this section we connect numerical choice to other prominent systems of higher-order RM, including Cousin's lemma (\cites{cousin1, dagsamIII}) and Kohlenbach's generalisation of weak K\"onig's lemma, called $\Phi_{1}$-$\WKL$ in \cite{kohlenbach4} .  
Our results include an elegant `splitting' as follows, assuming $\SIND$:
\be\label{coolsplit}
\ENC^{\omega}\asa [\BNC^{\omega}+\Phi_{1}\textup{-}\WKL]\asa [\BNC^{\omega}+\HBT]\asa  [\BNC^{\omega}+\Phi_{1}\textup{-}\KL], 
\ee
where the latter is $\Phi_{1}$-$\WKL$ generalised to finitely branching trees and where $\HBT$ is a version of Cousin's lemma (\cite{cousin1}).  We find these results particularly pleasing as \cite{kohlenbach4} constitutes one of the first papers related to higher-order RM.  

\smallskip

We also connect numerical choice to the coding principle $\open$ introduced in \cite{dagsamVII}, which expresses that open sets have RM-codes.      
The latter is among the strongest natural third-order principles in higher-order RM, and follows from $\ENC^{\omega}$ if we assume the basic fact that open sets are $\bf F_{\sigma}$ (Theorem \ref{finalt}). 

\subsection{Cousin's lemma}\label{notenoughleeb}
We connect $\ENC^{\omega}$ to Cousin's lemma and Heine-Borel compactness, via the Lebesgue number lemma.

\smallskip

First of all, Cousin's lemma is formulated as follows (\cite{dagsamIII, cousin1}), i.e.\ a version of the Heine-Borel theorem for uncountable coverings, whence its acronym.  
\begin{princ}[$\HBU$]
For any $\Psi:[0,1]\di \R^{+}$, there are $x_{0}, \dots, x_{m}\in [0,1]$ with $[0,1]\subseteq \cup_{i\leq m}B(x_{i}, \Psi(x_{i}))$.
\end{princ}
Over a reasonable base theory, $\HBU$ is equivalent to $\HBT$ (\cite{samBOOK, sahotop, dagsamXVII}) as follows.
\begin{princ}[$\HBT$]
For any $\psi:[0,1]\di \R^{+}\cup\{0\}$ such that for $(\forall x\in [0,1])(\exists y\in [0,1])(x\in B(y,\psi(y)))$, there are $x_{0}, \dots, x_{m}\in [0,1]$ with $[0,1]\subseteq \cup_{i\leq m}B(x_{i}, \psi(x_{i}))$.
\end{princ}
The coverings in $\HBT$ are clearly more general; in fact, the ones in $\HBU$ are unsuitable for general topology, like formulating the notion of dimension (see \cite{sahotop}). 

\smallskip

Secondly, we have the following result.  
\begin{thm}[$\ACAo+\SIND$] \label{ENC_to_HBT}
Numerical choice as in $\ENC^{\omega}$ implies Cousin's lemma as in $\HBT$.
\end{thm}
\begin{proof}
Fix $\psi:[0,1]\di \R^{+}\cup\{0\}$ such that $(\forall x\in [0,1])(\exists y\in [0,1])(x\in B(y,\psi(y)))$; 
we shall prove the associated Lebesgue number lemma as in \eqref{lnm} below, which clearly implies $\HBT$ for $\psi$ using $\SIND$:
\be\label{lnm}\textstyle
(\exists k\in \N)(\forall q, r\in \Q\cap [0,1])( |q-r|<\frac{1}{2^{k}}\di (\exists x\in [0,1])( [q, r]   \subset B(x, \psi(x))  ).  
\ee
Suppose the previous formula \eqref{lnm} is false, i.e.\ 
\be\label{lnm2}\textstyle
(\forall k\in \N)(\exists q, r\in \Q\cap [0,1])( |q-r|<\frac{1}{2^{k}}\wedge (\forall x\in [0,1])( [q, r]   \not\subset B(x, \psi(x))  ),
\ee
which has the right syntactical form to apply $\ENC^{\omega}$.  Let $g:\N\di \Q^{2}$ be a witnessing function for \eqref{lnm2} and define the sequence $x_{n}:= \frac{g(n)(1)+g(n)(2)}{2}$.  
Let $h\in \N^{\N}$ be such that $(x_{h(n)})_{n\in \N}$ is a convergent sub-sequence, say with limit $y\in [0,1]$, which is covered by $B(y_{0}, \psi(y_{0}))$ for some $y_{0}\in [0,1]$.   Clearly, $x_{h(n)}$ is eventually in $B(y_{0}, \psi(y_{0}))$, but also $(g(h(n))(1), g(h(n))(2)   )\subseteq B(y_{0}, \psi(y_{0}))$ as $|(g(h(n))(1)- g(h(n))(2)   |<\frac{1}{2^{n}}$.  This is a contradiction and we are done.  

\smallskip

Alternatively, for a more direct proof close to Cousin's original (\cite{cousin1}), suppose $\HBT$ is false for $\psi:[0,1]\di \R^{+}$ and 
consider 
\be\label{frowna}
(\forall n\in \N)(\exists i\leq 2^{n})(\forall w^{1^{*}})\big[ [i/2^{n}, (i+1)/2^{n}]\not\subset \cup_{j<|w|} B(w(j), \psi(w(j)))  \big], 
\ee
where $w^{1^{*}}$ is a variable for finite sequences of reals. 
Apply $\ENC^{\omega}$ to \eqref{frowna} to obtain $g\in \N^{\N}$ with $g(n)=i$ in \eqref{frowna}.  
Define $q_{n}:= \frac{g(n)}{2^{n}}$ and use sequential compactness to obtain a convergent sub-sequence, say with limit $y\in [0,1]$.  
However, $y\in B(y_{0}, \psi(y_{0}))$ for some $y_{0}\in [0,1]$, i.e.\ the latter sub-sequence (and the associated interval provided by \eqref{frowna}) is eventually in the former, a contradiction. 
\end{proof}
It should be noted that Theorem \ref{ENC_to_HBT} only needs `weak $\SIND$', as follows:
\[
(\forall m < m_0) ( \exists x \in \R)  (Y(m,x)=0) \di (\exists f:\N\di \R ) (\forall m<m_0) (Y(m,f(m))=0),
\]
for any $Y:\N\times \R\di \{0,1\}$ and $m_{0}\in \N$, which is called `$\IND_{2}$' in \cite{samBOOK}.  
Moreover, we only need the restriction of $\ENC^\omega$ to the case where upper bounds are available as a function; assuming $\SIND$, this restriction is equivalent to a generalisation of weak K\"onig's lemma due to Kohlenbach, studied in the next section. 

\subsection{Kohlenbach's generalisations of weak K\"onig's lemma}\label{kokolema}
We introduce a generalisation of weak K\"onig's lemma due to Kohlenbach from \cite{kohlenbach4}*{\S5} and connect it to Cousin's lemma as in $\HBT$ from Section \ref{notenoughleeb}.

\smallskip

First of all, Kohlenbach's versions of weak K\"onig's lemma allow for the tree-elementhood-formula `$\sigma\in T$' to be in a certain formula class beyond the quantifier-free.  
We shall study $\Phi_{1}$-$\WKL$ in this section, defined in Principle \ref{desney}.
Our formulation is equivalent to Kohlenbach's original version over $\ACAo$.  
\begin{princ}[$\Phi_{1}$-$\WKL$]\label{desney}
Let $A(\sigma)$ have the form $(\forall g\in 2^{\N})B(g, \sigma)$ with $B$ arithmetical. 
We have
\[
(\forall n\in \N)(\exists f\in 2^{\N})( \forall m\leq n )A(\overline{f}m)\di (\exists f\in 2^{\N})(\forall n\in \N)A(\overline{f}n)
\]
\end{princ}
We shall refer to the previous principle as \emph{Kohlenbach's weak lemma}, which is related to the neighbourhood function principle $\NFP$ (\cite{troeleke1}). 
We also need the following generalisation of $\QFAC^{1, 0}$ from the RM of $\NFP$ (\cite{samhabil, samBOOK}). 
\begin{princ}[$\A_0$]
For $A(\sigma)$ of the form $(\exists g\in 2^{\N})B(g, \sigma)$ with $B$ arithmetical:
\[
 (\forall f\in 2^{\N})(\exists n\in \N)A(\overline{f}n)\di  (\exists G^{2})(\forall f\in 2^{\N})(\exists n\leq G(f))A(\overline{f}n)
\]
\end{princ}
By the results in \cite{dagsamXVII} or \cite{samBOOK}*{Appendix}, $\Z_{2}^{\omega}+\QFAC^{0,1}+\A_{0}$ cannot prove $\NIN_{[0,1]}$ or $\HBU$, i.e.\ assuming $\A_{0}$ does not influence the hardness of the latter.  

\smallskip

Secondly, building on the previous section, we connect Cousin's lemma and Kohlenbach's weak lemma, and hence numerical choice.  
\begin{thm}[$\ACAo+\SIND$]\label{texno}
We have $\HBT\asa \Phi_{1}$-$\WKL\asa [\HBU+\A_{0}]$.
\end{thm}
\begin{proof}
Assume $\Phi_{1}$-$\WKL$ and let $\psi:[0,1]\di \R^{+}\cup\{0\}$ satisfy $(\forall x\in [0,1])(\exists y\in [0,1])(x\in B(y, \psi(y)) )$.  
Towards applying the contraposition of $\Phi_{1}$-$\WKL$, consider
\[\textstyle
(\forall f\in 2^{\N})(\exists n\in \N)( \exists g\in 2^{\N} )\big[ (\r(\overline{f}n*00)-\frac{1}{2^{n+1}}, \r(\overline{f}n*00)+\frac{1}{2^{n+1}})\subset B(\r(g), \psi(\r(g)))  \big].
\]
where $\r(f)=\sum_{n=0}^{\infty} \frac{f(n)}{2^{n+1}}$.  Then Kohlenbach's weak lemma provides an $n_{0}\in \N$ with $n\leq n_{0}$ in the previous formula.  
Then there is a finite sub-covering of $\cup_{x\in [0,1]}B(x, \psi(x))$ of at most $2^{n_{0}}$ elements.  To collect the associated reals into a finite sequence (based on the quantifier `$(\exists g\in 2^{\N})$'), use $\SIND$.
Hence, we obtain $\HBT$ while $\A_{0}$ trivially follows.  

\smallskip

Now assume $\HBU$ and let $A$ be as in $\Phi_{1}$-$\WKL$.  
Using standard coding of reals, one readily shows that $\HBU$ is equivalent to $\HBU_{2^{\N}}$ (\cite{dagsamIII}), i.e.\
\[
(\forall G^{2})(\exists f_{0}, \dots, f_{k}\in 2^{\N})( 2^{\N}\subset\cup_{i\leq k}[   \overline{f_{i}}G(f_{i})]  ).
\]
Assume $(\forall f\in 2^{\N})(\exists n\in \N)\neg A(\overline{f}n)$ and apply $\A_{0}$.  In turn, apply $\HBU_{2^{\N}}$ to the resulting $G^{2}$ and put $n_{0}=1+\max_{i\leq k}G(f_{i})$.  
Then $(\forall f\in 2^{\N})(\exists n\leq n_{0})\neg A(\overline{f}n)$, as required for $\Phi_{1}$-$\WKL$.  

\smallskip

Finally, assume $\HBT$ and let $B(f^{1}, \sigma^{0*})$ be as in $\Phi_{1}$-$\WKL$.  
Towards obtaining the contraposition of the latter, define $\psi:[0,1]\di \R^{+}\cup\{0\}$ as:
\[
\psi(x):=
\begin{cases}
\frac{1}{2^{n}} & \begin{array}{l}\textup{if $n$ is the least natural number satisfying} \\ \textup{the formula $\neg B(\eta(x)(n+2)* \eta(x)(n+3)*\dots, \overline{\eta(x)}n)$}\end{array} \\
0 & \textup{otherwise}
\end{cases}
\]
where $\eta:[0,1]\di 2^{\N}$ converts reals to binary, choosing a tail of zeros if necessary.  
Reals with multiple binary representations have computable descriptions, i.e.\ we may assume these cases do not occur.  
Then $\cup_{x\in [0,1]}B(x, \psi(x))$ covers the unit interval and any finite sub-covering readily yields $n_{0}\in \N$ with $(\forall f\in 2^{\N})(\exists n\leq n_{0})\neg A(\overline{f}n)$, as required.  
\end{proof}
The following corollary is a significant improvement of Kohlenbach's negative result regarding $\Phi_{1}$-$\WKL$ from \cite{kohlenbach4}*{Prop.\ 5.11}.
\begin{cor}
The system $\Z_{2}^{\omega}+\QFAC^{0,1}$ cannot prove $\Phi_{1}$-$\WKL$.
\end{cor}
\begin{proof}
As noted in Section \ref{ohoam}, the system from the corollary cannot prove $\NIN_{[0,1]}$.  
To show that $\Phi_{1}$-$\WKL$ implies the latter, let $Y:[0,1]\di \N$ be an injection.  
Apply $\HBU$ for $\Psi(x):=\frac{1}{2^{Y(x)+5}}$ and observe we obtain a contradiction. 
\end{proof}
A stronger result actually holds: $\Phi_{1}$-$\WKL$ proves that a countable set of reals can be enumerated.  

\smallskip

Finally, the above equivalences also go through \emph{mutatis mutandis} for the generalisation of $\WWKL$ from Remark \ref{trixie}, which we could dub $\Phi_{1}$-$\WWKL$, 
and Vitali's principle, called $\WHBU$ in \cite{dagsamVI, samBOOK}.  The latter (only) states the existence of finite coverings up to arbitrary small error. 
\subsection{A splitting for numerical choice}
In this section, we answer the following question in the positive.
\begin{center}
\emph{Is there a natural principle $\X$ such that $\ENC^{\omega}\asa [\X+\BNC^{\omega} ]$?}
\end{center}
Such results are called `splittings' and are fairly rare in second-order RM (see \cite{samsplit}).
To answer the previous question, we first generalise Kohlenbach's $\Phi_{1}$-$\WKL$ beyond binary trees, namely as follows.  
\begin{princ}[$\Phi_{1}$-$\KL$]
Let $A(\sigma)$ have the form $(\forall g\in 2^{\N})B(g, \sigma)$ with $B$ arithmetical. 
If $A$ is finitely branching\footnote{We say that $A(\sigma)$ is finitely branching if for any finite sequence $\sigma^{0^{*}}$, we have $A(\sigma)\di (\exists n\in \N)(\forall m\in \N)[A(\sigma*\langle m\rangle)\di m\leq n]$.  This is the same as the second-order definition \cite{simpson2}*{III.7.1}.}, then  
\be\label{KLZ}
(\forall n\in \N)(\exists f\in \N^{\N})( \forall m\leq n )A(\overline{f}m)\di (\exists f\in \N^{\N})(\forall n\in \N)A(\overline{f}n)
\ee
\end{princ}
\noindent 
We shall refer to the previous as `Kohlenbach's lemma'.  

\smallskip

We now have the following theorem that `splits' $\ENC^{\omega}$ into two natural parts.
\begin{thm}[$\ACAo+\SIND$]
We have  $\ENC^{\omega}\asa [\BNC^{\omega}+\Phi_{1}\textup{-}\KL]$.
\end{thm}
\begin{proof}
For the forward implication, let $A, B$ be as in $\Phi_{1}$-$\KL$ and apply $\ENC^{\omega}$ to 
\[
(\forall n\in \N)(\exists \sigma^{0^{*}}) ( \forall m\leq |\sigma| )A(\overline{\sigma}m)
\]
and let $g:\N\di \N^{<\N}$ be the resulting witnessing function.
Define a (usual second-order) tree by $\sigma \in T\asa (\exists n,m\in \N)( \overline{\sigma}n =\overline{g(m)} n)$.
Now apply the usual K\"onig's lemma to $T$, which is infinite and finitely branching by definition.  
The resulting path $f\in \N^{\N}$ in $T$ is such that $(\forall n\in \N)A(\overline{f}n)$, as required for \eqref{KLZ}. 

\smallskip

For the reverse implication, let $\varphi$ be arithmetical and such that $(\forall n\in \N)(\exists m\in \N)( \forall f\in 2^{\N})\varphi(f, m, n)$. 
Let $g\in \N^{\N}$ be the function provided by $\BNC^{\omega}$.  Now consider the formula $A(\sigma)$ as follows: 
\be\label{klaz}
(\forall i<|\sigma|)[\sigma(i)\leq g(i)\wedge ( \forall f\in 2^{\N})\varphi(f, \sigma(i), i) ].
\ee
By the induction provided, we obtain the antecedent of \eqref{KLZ}, while the finitely branching property follows from the bound provided by $g$. 
By $\Phi_{1}$-$\KL$, there is $h\in \N^{\N}$ with $(\forall n\in \N)A(\overline{h}n)$.  In light of \eqref{klaz}, $h$ yields the witnessing function required by $\ENC^{\omega}$, and we are done.
\end{proof}
A number of variations are possible, which we list in the following remark.
\begin{rem}\rm
Let $\ENC_{b}^{\omega}$ be the restriction of $\ENC^{\omega}$ to antecedents of the form $(\forall n\in \N)(\exists m\leq h(m))\varphi(n,m)$.
Let $\ENC_{\fin}^{\omega}$ be the restriction of $\ENC^{\omega}$ to antecedents of the form $(\forall n\in \N)(\exists m\in \N)\varphi(n,m)$ where 
there are at most finitely many (i.e. boundedly many) such $m\in \N$ for any fixed $n\in \N$.
Then it is straightforward to prove that 
\[
\ENC^{\omega}\asa  [\ENC_b^\omega + \BNC^\omega] \textup{ and }  \ENC^\omega\di  \ENC_\fin^\omega \di \ENC_b^\omega \di  \SSEP.
\]
Assuming $\SIND$, we also have $\ENC_{b}^{\omega}\asa \Phi_{1}$-$\WKL$.  Moreover, assuming the full axiom of induction, we have $\Phi_{1}$-$\KL\asa  \ENC_{\fin}^{\omega}$.   
Hence, \eqref{coolsplit} follows and we could replace $\HBT$ by e.g.\ the Lindel\"of lemma for the real line.  We note that Kohlenbach mentions $\Phi_{1}$-$\WKL\di \ENC^{\omega}$ in the proof of \cite{kohlenbach4}*{Prop.\ 5.11}. 
\end{rem}
Next, Kohlenbach also introduces $\Psi_{1}$-$\WKL$ in \cite{kohlenbach4}*{\S5}, which is $\Phi_{1}$-$\WKL$ with `$(\forall g\in 2^{\N})$' replaced by `$(\exists g \in 2^{\N})$' everywhere.  
In contrast to the latter, the former is rather weak, provable as it is in higher-order $\SAC$.  
\begin{thm}
The system $\ACAo+\QFAC^{0,1}$ proves $\Psi_{1}$-$\WKL$.
\end{thm}
\begin{proof}
Let $A, B$ as in $\Psi_{1}$-$\WKL$ be such that $(\forall n_{0}\in \N)(\exists f\in 2^{\N})(\forall n\leq n_{0})\neg A(\overline{f}n)$, 
where $\neg A(\sigma)$ is $(\exists g\in 2^{\N})B(\sigma, g)$.  Apply $\QFAC^{0,1}$ to the innermost formula to obtain a sequence $(g_{n})_{n\in \N}$ with
\be\label{netto}
(\forall n_{0}\in \N)\big(\exists \sigma\in 2^{<\N} \big)\big[|\sigma|=n_{0} \wedge  (\exists (g_{n})_{n\in \N})(\forall n\leq n_{0})\neg B (\overline{\sigma}n, g_{n})\big].
\ee
Apply $\QFAC^{0,1}$ once more to obtain a sequence $(\sigma_{n})_{n\in \N}$ such that $\sigma_{n_{0}}=\sigma$ for all $n_{0}\in \N$ in \eqref{netto}.  
Define a (second-order) infinite binary tree with this sequence and apply $\WKL$.  The resulting path $f_{0}\in 2^{\N}$ satisfies $(\forall n\in \N)\neg A(\overline{f_{0}}n)  $.
\end{proof}
One readily generalises $\Psi_{1}$-$\WKL$ to $\Psi_{1}$-$\KL$; the latter also follows from $\QFAC^{0,1}$. 
Since both versions of K\"onig's lemma follow from $(\exists^{3})$, a reversal is out of the question.  
We now discuss equivalences for $\Phi_{1}$-$\KL$ and related principles.
\begin{rem}\label{wassenaar}\rm
First of all, second-order K\"onig's lemma and Ramsey's theorem are equivalent, with direct proofs (not via $\ACA_{0}$) available (\cites{simpson2, moflo}).  
Assuming $\ACAo+\QFAC^{0,1}$, the same proofs yield an equivalence between $\Phi_{1}$-$\KL$ and Ramsey's theorem for $\Pi$-classes $X\subset \N$; the latter are such that `$n\in X$' is given by  $(\forall g\in 2^{\N})(Y(g, n)=0)$ for some $Y^{2}$.  
The crux is that the construction from \cite{moflo} can be done using $\QFAC^{0,1}$ if the homogenous sub-class is given as a sequence. 

\smallskip

Secondly, the principle $\ADS$ is well-known from the RM-zoo (\cite{damirzoo, dsliceke}) and states that in a countable linear ordering, there is either a descending or ascending sequence.  
Let $\Pi$-$\ADS$ be the former generalised to $\Pi$-classes.  One readily shows that this principle follows from Ramsey's theorem for two colours and pairs from $\Pi$-classes, if the homogenous sub-class is given as a sequence.  

\smallskip

Thirdly, we consider a generalisation of the countable Heine-Borel theorem as follows.  Fix arithmetical $B(\sigma^{0^{*}}, g)$ and define a closed $\Pi$-class $C\subset 2^{\N}$ by letting `$f\in C$' be short for $ (\forall n\in \N, g\in 2^{\N})B(\overline{f}n, g)$.
Then a sequence $(\sigma_{n})_{n\in \N}$ in $2^{<\N}$ covers $C$ in case $(\forall f\in 2^{\N})(\exists n\in \N)( f\in C\di f\in [\sigma_{n}] )$. 
Then $\Phi_{1}$-$\WKL$ implies that there is $n_{0}\in \N$ such that $(\sigma_{n})_{n\leq n_{0}}$ covers $C$.  The reversal is also easy.   
\end{rem}
The results in Remark \ref{wassenaar} are not particularly natural or surprising: they are statements that $\Pi$-classes have certain nice sub-sequences (rather than nice sub-classes).  
The equivalence $\Phi_{1}$-$\WKL\asa\HBT$ from Theorem \ref{texno} is different in kind in that no $\Pi$-classes are mentioned in $\HBT$ at all.  

\subsection{Open sets and their representations}
We connect numerical choice to the coding principle $\open$ from \cite{dagsamVII}, which simply expresses that open sets of reals are given by second-order codes.  

\smallskip

First of all, we introduce two fragments of $\open$.  To this end, we need the definition of R2-open set from \cite{dagsamVII}, which is a fairly strong representation. 
\bdefi
A set $O\subset \R$ is \emph{R2-open} if there is $Y:\R\di \R^{+}\cup\{0\}$ such that $x\in O\asa Y(x)>0$ and $B(x,Y(x) )\subset O$ for all $x\in \R$. 
\edefi
The following fragments, also studied in \cites{samBOOK, dagsamVII}, are needed for the below. 
\begin{defi}[Two fragments of $\open$]~
\begin{itemize}
\item $\open^{\dagger}$\textup{:} an open set $O\subset \R$ is $\bf F_{\sigma}$, i.e.\ there is an increasing sequence $(C_{n})_{n\in \N}$ of closed sets such that $O=\cup_{n\in \N}C_{n}$
\item $\open^{-}$\textup{:} an R2-open $O\subset \R$ is RM-open, i.e.\ there are sequences of reals $(a_{n})_{n\in \N}$ and $(b_{n})_{n\in \N}$ such that $O=\cup_{n\in \N}(a_{n}, b_{n})$. 
\end{itemize}
\end{defi}
We stress that `$\bf F_{\sigma}$' refers to the usual definition, i.e.\ not involving RM-codes.  

\smallskip

Secondly, we have the following theorem. 
Note that $\open^{\dagger}$ is weak in that it does not\footnote{Observe that $\neg \NIN_{[0,1]}$ implies $\open^{\dagger}$, as in the former case \emph{all} sets are countable and hence trivially $\bf F_{\sigma}$.  Thus, if $\open^{\dagger}$ implies $\NIN_{[0,1]}$ in say $\ACAo$, we can prove $\NIN_{[0,1]}$ in the latter system by invoking the law of excluded middle for $\open^{\dagger}$, a contradiction.} imply $\NIN_{[0,1]}$.  
Observe that sub-item (c.2) is an instance of numerical choice restricted to formulas about closed sets. 
\begin{thm}[$\ACAo+\QFAC^{0,1}$]\label{finalt} 
The higher items imply the lower ones.
\begin{enumerate}
\renewcommand{\theenumi}{\alph{enumi}}
\item The combination of the following items. 
 \begin{itemize}
\item[(a.1)] The principle $\open^{\dagger}$.  
\item[(a.2)] Numerical choice as in $\ENC^{\omega}$
\end{itemize}
\item The combination of the following items. 
 \begin{itemize}
\item[(b.1)] The principle $\open^{\dagger}$.
\item[(b.2)] Numerical choice as in $\BNC^{\omega}$
\item[(b.3)] The principle $\open^{-}$. 
\end{itemize}
\item The combination of the following items. \label{xexs}
 \begin{itemize}
\item[(c.1)] The principle $\open^{\dagger}$.
\item[(c.2)] Let $(C_{n})_{n\in \N}$, $(D_{n})_{n\in \N}$ be non-empty closed sets in $[0,1]$ with $(\forall n\in \N)(C_{n}\cap D_{n}= \emptyset)$.  There is $g\in \N^{\N}$ with $(\forall n\in \N)( d(C_{n}, D_{n})>\frac{1}{2^{g(n)}})$.
\item[(c.3)] The principle $\open^{-}$. 
\end{itemize}
\item The coding principle $\open$\textup{:} every open set $O\subset \R$ is RM-open.\label{xexs2} 
\end{enumerate}
\end{thm}
Items \eqref{xexs} and \eqref{xexs2} are equivalent.  
\begin{proof}
We work in the unit interval and observe that all results readily generalise to the reals.  
Let $O$ be open and let $(F_{n})_{n\in \N}$ be a sequence of closed sets such that $O=\cup_{n\in \N}F_{n}$ and $F_{n}\subseteq F_{n+1}$ for all $n\in \N$.
Let $C$ be the complement of $O$ and note that $C$ and each $F_{n}$ are disjoint.  In fact, we have 
\be\label{texnico}\textstyle
(\forall n\in \N)(\exists k\in \N)(\forall x\in F_{n}, y\in C)(|x-y|>\frac{1}{2^{k}})
\ee
which readily follows via the obvious proof-by-contradiction (using $\QFAC^{0,1}$).  Now apply item (c.2) to obtain $g\in \N^{\N}$ such that $g(n)\geq k$ in \eqref{texnico}. 
Next, define $Y:[0,1]\di \R$ as follows:  $Y(x):=0$ if $x\in C$ and $Y(x)=\frac{1}{2^{h(x)+1}}$ where $h(x):= g((\mu n)( x\not\in F_{n+1}  )-1)$. 
By definition, we have $B(x, Y(x))\subset O$ if $x\in O$, i.e.\ $Y$ is an R2-representation of $O$. 
Note that $\HBU$ implies that R2-open sets are RM-open by \cite{dagsamVII}*{Theorem 7.8}, while the former follows from $\ENC^{\omega}$ by Theorem \ref{ENC_to_HBT}.
Hence, all downward implications now follow.  

\smallskip

For the reversal, the sub-items of item \eqref{xexs} are provable in $\ACAo$ for RM-codes.  Indeed, the distance function $d(x, C)$ exists for RM-closed sets $C\subset \R$ by \cite{withgusto}*{Theorem 1.2} and the usual second-order proofs then go through.  For instance, to show that an RM-closed set $C\subset [0,1]$ is $\bf G_{\delta}$, define the open set $G_{n}:= \{x\in [0,1] : x\in C \vee d(x, C)<\frac{1}{2^{n}}\}$ and observe that $C=\cap_{n\in \N}G_{n}$.  For item (c.2), RM-closed sets are \emph{separably closed} in $\ACA_{0}$ (\cite{withgusto}), where the latter means that a dense sub-\emph{sequence} is given.  
Hence, the formula $(\forall n\in \N)(\exists m\in \N)(d(C_{n}, D_{n})>\frac{1}{2^{m}})$ is equivalent to $\L_{2}$-arithmetical. 
\end{proof}
Assuming the weak $\open^{\dagger}$, numerical choice as in $\BNC^{\omega}$ even implies $\BCT_{[0,1]}$, the Baire category theorem for the unit interval, a `Big' system by \cite{samBIG2}. 
\begin{cor}[$\ACAo$]\label{tankpilk}
$[\BNC^{\omega}+\open^{\dagger}]\di\BCT_{[0,1]}$, where the latter states that for a decreasing sequence of dense open sets in $[0,1]$, the intersection is non-empty. 
\end{cor}
\begin{proof}
By the proof of Theorem \ref{finalt}, $\BNC^{\omega}+\open^{\dagger}$ implies that every open set of reals is R2-open.
However, $\BCT_{[0,1]}$ is provable in $\ACAo$ for R2-open sets, using the usual constructive proof from \cite{bish1}*{p.~87} (see also \cite{samBOOK}*{D.4.1} or \cite{dagsamVII}*{\S7.3}).  
\end{proof}
Finally, Theorem \ref{finaility} connects $\open$, numerical choice, and Lindel\"of's lemma (\cite{blindeloef}) in an elegant way.  
The `open choice' principle $\OC^{0,0}$ was first studied in \cite{dagsamXVI} and clearly follows from $\BNC^{\omega}$. 
\begin{princ}[${\OC}^{0,0}$] 
For any increasing sequence of open sets $(O_{n})_{n\in \N}$ in $\R$, we have the following:
\[
(\forall n\in \N)(\exists m\in \N)([-n, n]\subset O_{m})\di (\exists g\in \N^{\N})(\forall n\in \N)([-n, n]\subset O_{g(n)}).
\]
\end{princ}
We conjecture that $\open$ does not imply $\HBU$, say over $\Z_{2}^{\omega}+\QFAC^{0,1}$. 
\begin{thm}[$\ACAo+\QFAC^{0,1}$]\label{finaility}
The higher items imply the lower ones.  
\begin{enumerate}
\renewcommand{\theenumi}{\alph{enumi}}
\item The combination of $\ENC^{\omega}$ and $\open^{\dagger}$.\label{vla1}
\item The combination of $\BNC^{\omega}$, $\HBU$, and $\open^{\dagger}$.\label{vla}
\item The combination of the following. \label{vlb}
\begin{itemize}
\item  \(Lindel\"of\) For any closed set $C\subset \R$ and any function $\Psi:\R\di \R^{+}$, there is a sequence $(x_{n})_{n\in \N}$ in $C$ such that $C\subset \cup_{n\in \N}B(x_{n}, \Psi(x_{n}))$. 
\item The principle $\OC^{0,0}$. 
\end{itemize}
\item Item \eqref{vlb} restricted to continuous $\Psi$.\label{vlc}
\item The coding principle $\open$. 
\end{enumerate}
\end{thm}
\begin{proof}
We recall that $\open^{-}$ follows from $\HBU$ by \cite{dagsamVII}*{Theorem 7.8}. 
The first downward implication follows from Theorem \ref{ENC_to_HBT}.
Assume item \eqref{vla} and fix closed $C\subset \R$ and any $\Psi:\R\di \R^{+}$.
By the proof of Theorem~\ref{finalt}, open sets now come with an R2-representation.  
Let $Y_{O}:\R\di \R$ be an R2-representation of $O:=\R\setminus C$.  
Define $\Psi_{0}:\R\di \R^{+}$ as follows: $\Psi_{0}(x):=\Psi(x)$ if $x\in C$ and as $Y_{O}(x)$ otherwise.  
Now, $\HBU$ implies item \eqref{vlb} for $C=\R$ assuming $\QFAC^{0,1}$ (to put all finite sub-coverings in a sequence).  
Apply the latter for $\Psi_{0}$ and observe that the first sub-item of item~\eqref{vlb} immediately follows. 
The second sub-item is readily proved in $\ACAo$ for RM-codes while item \eqref{vla} implies $\open$ by (the proof of) Theorem \ref{finalt}.  

\smallskip

Next, assume item \eqref{vlc} and observe that $\open$ restricted to the unit interval readily implies $\open$.  
In this light, fix closed $C\subset [0,1]$, and let $D\subset \R$ be set obtained by placing copies of $C$ on each $[2n, 2n+1]$.  
Define $\Psi(x):=\frac{1}{2^{n}}$ on $[2n, 2n+1]$ and linear on the other intervals.  
Now apply Lindel\"of's lemma for $D$ and $\Psi$ and let $(x_{m})_{m\in \N}$ be the associated sequence.  
Then $\cup_{m\in \N}B(x_{m}, \frac{1}{2^{n}})$ covers the copy of $C$ on $[2n, 2n+1]$.  The countable Heine-Borel theorem (provable using $\QFAC^{0,1}$ and the obvious proof-by-contradiction) implies 
there is a finite sub-covering.  By \cite{dagsamXIV}*{Theorem 12}, $\OC^{0,0}$ implies the sequential version of this version of Heine-Borel theorem.  
The sequence of these finite sub-coverings yields a code for $O$. 
\end{proof}
The following theorem identifies item (c.2) from Theorem \ref{finalt} as the `strong' sub-item as it still implies $\ATR_{0}$. 
\begin{thm}[$\ACAo$]
Item \textup{(c.2)} from Theorem \ref{finalt} implies $\ATR_{0}$.
\end{thm}
\begin{proof}
We derive item \eqref{kuhl} from Theorem \ref{hofff} and hence $\ATR_0$.  
Let $f:\R\di \R$ be as in the former item \eqref{kuhl} and consider the set 
\[
A_{n}:=\{ (x, y)\in \R: x\in [n, n+1]\wedge 0\leq y\leq f(x)   \}  .
\]  
Now perform a one-point compactification, and then use reasonable coding (or a space-filling curve) to transfer $A_{n}$ to $[0,1]$.  Then take the closure of the resulting set, finally yielding a closed set $C_{n}$.  
We let $D_{n}$ be the singleton point that was the point at infinity (before coding).  
A lower bound on the closeness of $C_{n}$ and $D_{n}$ gives an upper bound for $f$, as required. 
\end{proof}
In conclusion, $\open$ is a surprisingly strong natural third-order principle not provable in $\Z_{2}^{\omega}+\QFAC^{0,1}$. 
By Theorem \ref{finalt}, $\ENC^{\omega}$ actually implies $\open$, assuming the rather weak principle $\open^{\dagger}$, i.e.\ that open sets are $\bf F_{\sigma}$. 
Moreover, numerical choice and $\open$ now seem similar, if not intimately related, in light of the former.

\section{Conclusion}\label{conculs}
In this final section, we provide a graphical overview with some foundational discussion (Section \ref{grafo}), and suggestions for future research (Section \ref{dasfutur}).

\subsection{Future research}\label{dasfutur}
We briefly discuss possible future research. 

\smallskip

First of all, we wish to generalise the conservation results in Theorem \ref{hunterX} to other principles, including the following:
\begin{itemize}
\item existence of arbitrary $\mathbf{\Delta^1_1}$ (i.e. Borel, assuming $\ATR_0$ ) sets of reals,
\item $\ENC$ and $\BNC$ generalised to $(\forall x\in \R)(\exists n\in \N)\varphi(x, n)$ with $\varphi\in \Pi_{1}^{1}$,
\item the generalisation of $\WWKL$ from Remark \ref{trixie}.
\end{itemize}
Secondly, we wish to explore the connection between $\BNC^{\omega}$ and e.g.\ the statement that countable sets of reals have measure zero.
The same for the theorems mentioned in Remarks \ref{bipoli} and \ref{doooonggg}.

\smallskip

Thirdly, we wish to generalise our results from metric spaces to (second-countable) topological spaces (\cite{samHYP, samHARD}).   
Since the latter are more general than the former, can we use more basic properties?

\smallskip

Fourth, the RM of $\WKL_{0}$, $\ACA_{0}$, and $\ATR_{0}$ should yield equivalences for $\Phi_{1}$-$\WKL$ and $\Phi_{1}$-$\KL$ following Remark \ref{wassenaar}.  
While the examples in the latter are somewhat natural, results not mentioning $\Pi$-classes are preferrable.  

\smallskip

Fifth, the splitting of $\ENC^{\omega}$ in \eqref{coolsplit} is quite pleasing and we would welcome similar results, especially for the other principles from Figure \ref{xxp}. 

\smallskip

Sixth, what is the role of the principle $\open^{\dagger}$ in RM?  Is it somehow connected to the Baire category theorem or Tao's pigeon hole principle for measure (\cite{samBIG2})?

\subsection{A graphical overview}\label{grafo}
The diagram in Figure \ref{xxp} provides an overview of the relations among the principles that have been studied the most in the above.  
Proofs may be found in the relevant references and \cite{samBOOK}. 

\smallskip

The higher-order principles of a given row plus $\ACAo$ are conservative extensions of the systems in the final column.  
Implications sometimes take place assuming extra induction or choice. 
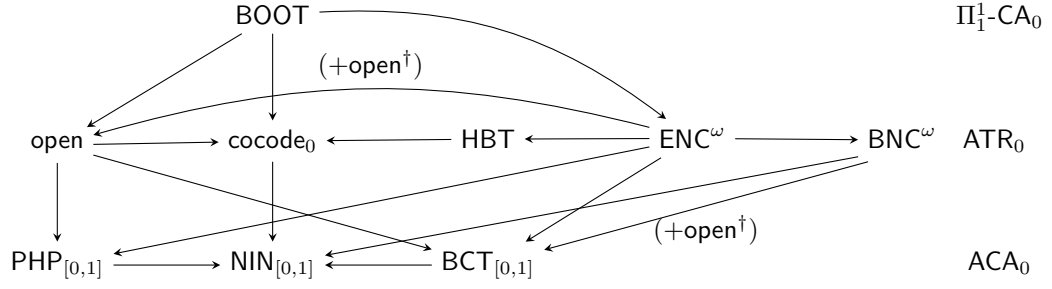
\begin{figure}[H]
\begin{tikzpicture}
  \matrix (m) [matrix of math nodes,row sep=3em,column sep=4em,minimum width=2em]
  {  
  ~ & \BOOT  & ~ &  ~&\qquad\qquad\FIVE \\
  \open & \cocode_{0}  & \HBT & \ENC^{\omega} &\BNC^{\omega}\quad  \ATR_{0} \\
    \PHP_{[0,1]} & \NIN_{[0,1]}  & \BCT_{[0,1]} & ~& \qquad\qquad\ACA_{0} \\            
        };
  \path[-stealth]
     (m-2-4) edge[bend right=15] node [above] {$(+\open^{\dagger})$} (m-2-1)
    (m-1-2) edge node [left] {} (m-2-1)
    (m-1-2) edge[bend left=20] node [left] {} (m-2-4)
    (m-2-4) edge node [left] {} (m-2-3)
    (m-2-3) edge node [left] {} (m-2-2)
    (m-2-1) edge node [left] {} (m-2-2)
    (m-1-2) edge node [left] {} (m-2-2)
    (m-2-4) edge node [left] {} (m-3-3)
    (m-3-3) edge node [left] {} (m-3-2)
    (m-3-1) edge node [left] {} (m-3-2)
    (m-2-1) edge node [left] {} (m-3-1)
    (m-2-1) edge node [left] {} (m-3-3)
    (m-2-2) edge node [left] {} (m-3-2)
    (m-2-4) edge node [left] {} (m-3-1)
        (m-2-5) edge node [left] {} (m-3-2)
                            (m-2-5) edge node [below] {$(+\open^{\dagger})$} (m-3-3)
        (m-2-4) edge node [left] {} (m-2-5)
;
\end{tikzpicture}
\caption{Some relations among our principles}
\label{xxp}
\end{figure}
\noindent
The various principles from Figure \ref{xxp} are defined as in the following list.
\begin{itemize}
\item $\BOOT$ (Feferman's projection principle, \cite{littlefef}): comprehension for $\Sigma$-formulas. 
\item $\open$ (\cite{dagsamVII}): an open set of reals can be written as a countable union of open intervals, 
\item $\cocode_{0}$: a countable set of reals can be enumerated as a sequence. 
\item $\PHP_{[0,1]}$ (Tao's pigeon hole principle for measure, \cite{samBIG2, taomes}): for a sequence of closed sets of reals of measure zero, the union also has measure zero. 
\item $\BCT_{[0,1]}$ (Baire category theorem, \cite{samBIG2}): for a sequence of open and dense sets of reals, the intersection is non-empty. 
\item $\open^{\dagger}$: an open set of reals is $\bf F_{\sigma}$. 
\end{itemize}
We also discuss some foundational matters as follows. 
\begin{rem}[Foundational musings]\rm
In higher-order RM, there are at least two key parameters that govern logical strength.  
\begin{itemize}  
\item[(A)] The complexity of the sets of reals that are guaranteed to exist
\item[(B)] Given a set of reals $X$, the complexity relative to $X$ of reals that are guaranteed to exist.  
\end{itemize}
The system $\Z_{2}^{\omega}$ is strong regarding (A) and very weak regarding (B), hence yielding unprovability of $\NIN_{[0,1]}$ and the like.  
Adding strength to both parameters (A) and (B) can lead to what we call explosive results, like $\ACAo+\HBU$ implying $\ATR_{0}$ while both $\ACAo$ and $\RCAo+\HBU$ do not exceed $\ACA_{0}$ qua second-order theorems  --  or  from Theorem $\ref{slif}$ , $\FIVE^{\omega}+\ENC^{\omega}$ proves $\SIX$ .   

\smallskip

Now principles like $\BNC^{\omega}$ and $\ENC^{\omega}$ are part of the RM of (B).  
Having both (A) and (B) (plus weakness in (B)) prevents third-order RM from having a small number of systems to which most theorems are equivalent.  In addition, even with strong (A), descriptive properties of the most complicated sets of reals included in a model can vary widely, hence relationships often form a complex mesh, like $[C \wedge D] \asa E$, instead of just a natural linear progression in complexity strengths.
\end{rem}

\begin{ack}\rm
We thank Anil Nerode for pushing us in the direction of the necessary and sufficient conditions for Riemann integrability, which was the starting point of this paper.  
Parts of this paper are motivated by the Habilitation of the first author at TU Darmstadt (\cite{samhabil}).
 We thank Ulrich Kohlenbach for his guidance in this regard. 
\end{ack}

\end{document}